\documentclass{amsart}

\usepackage{amssymb}
\usepackage[dvips]{epsfig}
\usepackage{graphicx}
\usepackage{color}
\usepackage{amsmath}
\usepackage{amsthm}
\usepackage{amssymb}
\usepackage{amsfonts}
\usepackage[applemac]{inputenc}
\usepackage{mathrsfs}

\usepackage{pdfsync}

\newcommand{\dsp}{\displaystyle}

\newcommand{\R}{{\mathbb R}}
\newcommand{\N}{{\mathbb N}}

\newcommand{\eps}{\varepsilon}

\newcommand{\boE}{\mathcal{E}}

\newcommand{\re}{\mathrm{Re}}
\newcommand{\im}{\mathrm{Im}}

\theoremstyle{plain}
\newtheorem{theorem}{Theorem}[section]
\newtheorem{lemma}[theorem]{Lemma}
\newtheorem{corollary}[theorem]{Corollary}
\newtheorem{proposition}[theorem]{Proposition}

\theoremstyle{definition}
\newtheorem{definition}[theorem]{Definition}

\newtheorem{remark}[theorem]{Remark}
\newtheorem*{remark*}{Remark}

\numberwithin{equation}{section}
\newtheorem*{merci}{Acknowledgements}

\theoremstyle{plain}

\numberwithin{equation}{section}

\begin{document}

\title[Dispersive blow up for NLS]{Dispersive blow up for nonlinear Schr\"odinger equations revisited}

\author{J. L. Bona}
\address{Department of Mathematics, Statistics and Computer Science\\
The University of Illinois at Chicago\\
851 S. Morgan Street MC249\\
Chicago, Illinois 60607-7045}
\email{bona@math.uic.edu}

\author{G. Ponce}
\address{Department of Mathematics, \\
University of California\\
Santa Barbara, California 93106}
\email{ponce@math.ucsb.edu}

\author{J.-C.  Saut}
\address{Laboratoire de Math\' ematiques, UMR 8628\\
Universit\' e Paris-Sud et CNRS\\ 91405 Orsay, France}
\email{jean-claude.saut@math.u-psud.fr}

\author{C. Sparber}
\address{Department of Mathematics, Statistics and Computer Science\\
The University of Illinois at Chicago\\
851 S. Morgan Street MC249\\
Chicago, Ilinois 60607-7045}
\email{sparber@math.uic.edu}

\begin{abstract}
The possibility of finite-time, dispersive blow up for nonlinear equations of Schr\"odinger type
is revisited.   
This mathematical phenomena is one of the conceivable explanations for oceanic 
and optical rogue waves. 
In dimension one, the fact that dispersive blow up does occur 
for  nonlinear Schr\"odinger equations already appears in \cite{BS2}.
In the present work, the existing results are extended in several ways.  In one direction, the theory is
broadened to include  the Davey-Stewartson and Gross-Pitaevskii equations.  In another, 
 dispersive blow up is shown to obtain for nonlinear Schr\"odinger equations  in spatial dimensions larger than one and for 
 more general power-law nonlinearities.  As a by-product of our analysis,  a sharp global smoothing estimate for the  integral term appearing in Duhamel's formula is obtained.

\vspace{.5cm}
\noindent {\scshape R\' esum\' e}.  Nous revisitons la possibilit\'e de la formation de singularit\' es
dispersives (dispersive blow-up) pour des solutions d'\' equations de
Schr\"{o}dinger non lin\' eaires. Ce ph\' enom\`ene math\' ematique pourrait \^{e}tre une
explication pour l'apparition des ``vagues sc\' el\' erates" (rogue waves) en
oc\' eanographie et optique non lin\' eaire.  L'\'emergence de  singularit\' es
dispersives pour des \' equations de Schr\"{o}dinger non lin\' eaires en dimension
spatiale un a \' et\' e  prouv\' ee dans \cite{BS2}. Ces r\' esultats sont \' etendus ici dans
plusieurs directions. D'une part, la th\' eorie est \' etendue \`a des \' equations de
Schr\"{o}dinger en dimension spatiale quelconque, avec des non-lin\' earit\' es de
type puissance g\' en\' erales. D'autre part, nous traitons \' egalement le cas des
syst\`emes de Davey-Stewartson et de l'\' equation de Gross-Pitaevskii. 
Un sous-produit de notre analyse est un
effet de lissage global pr\' ecis pour le terme int\' egral de la repr\' esentation
de Duhamel.
\end{abstract}

\date{\today}


\thanks{We thank F. Linares for a careful reading of the manuscript and Remi Carles for 
helpful conversation and suggestions.  Part of this work was 
done while J. B. was a visiting professor at the Universit\'e Bordeaux and at the 
Universit\'e Paris Nord.
G.~P. and C.~S. acknowledge support of the NSF through grants DMS-1101499 and  DMS-1161580, respectively. Additional support was provided 
through the NSF research network Ki-Net. J.-C. S. acknowledges partial support from the project GEODISP of the ANR}

\maketitle

\section{Introduction}

This paper continues the theory of dispersive blow up which was 
initiated and developed  in \cite{BS1} and \cite{BS2}.  The present contribution is especially relevant to nonlinear 
Schr\"odinger-type equations, and includes theory for the Davey-Stewartson and the Gross-Pitaevskii equations. The work on 
 Schr\"odinger equations substantially extends the results already available in \cite{BS2}.  

Dispersive blow up of wave equations is a phenomenon of focusing of 
smooth initial disturbances with finite-mass (or, finite-energy, depending on the physical context) that 
relies upon the dispersion relation guaranteeing that, in the
linear regime, different
wavelengths propagate at different speeds.   This is especially the 
case for models wherein the linear dispersion is unbounded, so that energy can
be moved around at arbitrarily high speeds, but even bounded dispersion
can exhibit this type of singularity formation.  

To be more concrete, consider the Cauchy problem 
for the linear (free) Schr\"odinger equation
\begin{equation}\label{eq:linSch}
i \partial_t u + \Delta u = 0, \quad u\big|_{t=0} = u_0(x),
\end{equation}
where $x\in \R^n$ for some $n\in \N$.  For $u_0\in L^2(\R^n)$, elementary 
Fourier analysis shows the solution to this initial-value problem is
\begin{equation}\label{eq:formula}
u(x,t) = e^{i t \Delta} u_0(x):= \frac{1}{(2\pi)^n} \int_{\R^n} e^{-i t |\xi|^2} \widehat u_0(\xi) e^{i \xi\cdot x} d\xi.
\end{equation}
Here, $\widehat u_0$ denotes the Fourier transformed initial data, {\it viz}.  
\[
\mathcal Fu_0(\xi) \equiv\widehat u_0(\xi)= \int_{\R^n} u_0(x) e^{-i \xi \cdot x} \, dx.
\]
The corresponding inverse Fourier transform will be denoted by $\mathcal F^{-1}$. From \eqref{eq:formula}, it is immediately inferred that for any $s \in \R$, solutions lie in $C(\R;H^s)$ 
whenever $u_0$ lies in the $L^2$-based Sobolev space $H^s$.  
Moreover,  the evolution  preserves all these Sobolev-space   
norms, which is to say 
\[\| u(\cdot, t) \|_{H^s(\R^n)} = \| u_0 \|_{H^s(R^n)} \] 
for $t\in \R$.   In certain applications of this model, the case $s=0$ in the last formula
 corresponds to 
conservation of total mass in the underlying physical system.

However, in Theorem 2.1 of \cite{BS2}, it was shown that for any given  point $(x_*, t_*)\in  \R^{n}\times \R_+$, 
there exists initial data  $u_0\in C^\infty(\R^n)\cap L^2(\R^n)\cap L^\infty(\R^n)$ 
such that the solution $u(x,t)$ of the corresponding initial-value problem \eqref{eq:linSch} 
for the free Schr\"odinger equation  is continuous on $\R^n\times \R_+ \setminus \{(x_*, t_*)\}$, but
\[
\lim _{(x,t)\in \R^{n}\times \R_+\to (x_*, t_*)} | u(x,t)| = + \infty.
\] 
This fact is referred to as (finite-time) dispersive blow up and will 
sometimes be abbreviated  DBU in the following. 
The analogous phenomena also appears in other linear dispersive equations, such 
as the linear  Korteweg-de Vries equation \cite{BBM} and the  free surface 
water waves system 
linearized around the rest  state \cite{BS2}.

At first sight, one would expect that nonlinear terms would destroy 
dispersive blow up.  What is a little surprising is that even the 
inclusion of physically relevant nonlinearities in various models 
of wave propagation does not prevent dispersive blow up.  Indeed, 
theory shows in some important cases that initial data 
leading  to this focusing singularity under the linear evolution continues to 
blow up in exactly the same way when nonlinear terms are 
included.  
In \cite{BS1}, this was shown to be 
true for the Korteweg-de Vries equation, a model for shallow water waves and other 
simple wave phenomena.  This result and
analogous dispersive blow up theory in \cite{BS2} for  
solutions of the one-dimensional nonlinear Schr\"odinger equations,
\begin{equation}
\label{**}
i \partial_t u + \partial^2_{x} u \pm |u|^pu=0, \quad u\big|_{t=0} = u_0(x),
\end{equation}
where $x\in \R$ and $p \in (0,3)$,
 lead to the 
speculation that such focusing might be one road to the formation of rogue waves 
in shallow and deep water and in nonlinear optics   
(see \cite{D, DGE, KPS, SRKJ}).  

The analysis put forward in \cite{BS1} and \cite{BS2} revolves around providing bounds on
the nonlinear terms in a Duhamel representation of the 
evolution. Because the phenomenon is due to the linear terms in the equation,
data of arbitrarily small size will still exhibit dispersive blow up, and
indeed it can be organized to happen arbitrarily quickly.   This 
emphasizes the linear aspect of these singularities and differentiates
it from the  blow up that occurs for some of the same
models when the nonlinear term is focusing and sufficiently strong (see \cite{SuSu} for a general overview
 of this aspect of Schr\"odinger equations). Moreover, even though the theory begins by showing that there are specific initial data that lead to 
dispersive blow up, the result is in fact self-improving.  Dispersive blow up
 continues to hold if this special initial data is subjected to a smooth perturbation. 
 The theory further implies that there is $C^\infty$ initial data with compact 
support which can be taken as small as we like that will, in finite time, become large 
in a neighborhood of a prescribed spatial point (see Remark 3.5 for more details). 
Dispersive blow up thereby also serves to demonstrate ill-posedness of the considered models in $L^\infty$--spaces.

The aim of the present work is to generalize the results mentioned above in several respects. 
Most importantly, the dispersive blow up that in \cite{BS2} was obtained for \eqref{**}  will be shown to
hold true of nonlinear Schr\"odinger 
equations in {\it all dimensions} $n \geq 1$ and for 
the whole range of nonlinearities $p\geq \lfloor\frac{n}{2}\rfloor$, with or without a (possibly unbounded) real-valued
potential.  Here and in the following, for $\mu \in \R$, the quantity $ \lfloor\mu \rfloor$ is
 the greatest integer less than or equal
to $\mu$.  Higher-order Schr\"odinger equations are also countenanced.  Our theory relies 
especially on the results of Cazenave and Weissler established in \cite{CW}. 
In addition to  Schr\"odinger equations, dispersive 
 blow up is proved for  Gross-Pitaevskii equations 
 with non-trivial boundary conditions at infinity and for the Davey-Stewartson systems.

As a by-product of our analysis, a sharp global smoothing effect is obtained 
for the nonlinear integral term in 
the  equation derived from \eqref{**} by use of Duhamel's formula.

The paper proceeds as follows:   Section 2 is concerned with some preliminaries 
which are mostly linear in nature.  Dispersive blow up for nonlinear Schr\"odinger 
equations is tackled in Section 3. Section 4 deals with the sharp 
global smoothing property mentioned earlier.  This latter theory, in addition to being 
of interest in itself, is used to complete the analysis in Section 3.   
Dispersive blow up for the Davey-Stewartson 
systems then follows more or less as a corollary to the results in Sections 2 and 3.
  The Gross-Pitaevskii equation takes center stage in 
Section 5 whilst higher-order Schr\"odinger equations are studied in Section 6.


\section{Mathematical preliminaries} \label{sec:prelim}

In this section, a review of the basic idea behind dispersive blow up 
is provided in the context of nonlinear
Schr\"odinger equations.  Parts of the currently available
 theory for the linear Schr\"odinger group 
are also recalled in preparation for the analysis in Section 3.   


\subsection{Dispersive blow up in linear Schr\"odinger equations} \label{sec:lin}

To understand the appearance of dispersive blow up in the solution of \eqref{eq:linSch},
start by explicitly computing the inverse Fourier transformation in \eqref{eq:formula} to
see that the free Schr\"odinger group admits the 
representation 
\begin{equation}\label{eq:group}
u(x,t)  = \frac{1}{(4i \pi t)^{n/2}} \int_{\R^n} e^{i\frac{|x-y|^2}{4t}} u_0(y)\, dy,\quad \text{for $t\not =0.$}
\end{equation}
This representation formula is the starting point of the following lemma.
\begin{lemma}\label{lem:IC}
Let $\alpha \in \R$, $q\in \R^n$ and 
\begin{equation*} 
u_0(x) = \frac{e^{- i  \alpha |x - q|^2}}{(1+|x|^2)^m}\ ,\quad \text{with $\frac{n}{4} < m \le \frac{n}{2}$.}
\end{equation*}
Then, $u_0 \in C^\infty(\R^n)\cap L^2(\R^n)\cap L^\infty(\R^n)$ and the associated global in-time solution 
$u \in C(\R; L^2(\R^n))$ of \eqref{eq:linSch} has the following properties.
\begin{enumerate}
\item At the point $(x_*,t_*)=(q,\frac{1}{4\alpha})$, the solution $u$ in \eqref{eq:group} blows up, 
which is to say, 
\[
\lim _{(x,t)\in \R^{n+1}\to (x_*,t_*)} | u(x,t)| = + \infty,
\]
\item it is a continuous function of $(x,t)$ on  $\, \R^n\times \R \setminus \{t_*\}$ and
\item  $u(x,t_0)$ is a continuous function of $x \in \R^n \setminus \{x_*\}$.
\end{enumerate} 
\end{lemma}
\begin{proof}   First note that $u_0 \in C^\infty(\R^n)\cap L^2(\R^n)\cap L^\infty(\R^n)$, that 
$u \in C(\R; L^2(\R^n))$ and that the $L^2$--norm of $u$ is constant in view of mass conservation. On the other hand, 
 evaluating \eqref{eq:group} at $t=\frac{1}{4\alpha}$ for this particular $u_0$   gives
\[
u\left(x, \frac{1}{4\alpha}\right)=\left( \frac{\alpha}{i\pi} \right)^{n/2} e^{i \alpha (|x|^2-|q|^2)} \int_{\R^n} e^{-2i \alpha y\cdot (x-q)} \frac{dy}{(1+|y|^2)^m}.
\]
Thus at $x=q$, it transpires that 
\[
\left| u\left(q, \frac{1}{4\alpha}\right) \right | =  \left( \frac{\alpha}{\pi} \right)^{n/2}   \int_{\R^n}  \frac{dy}{(1+|y|^2)^m}  = +\infty
\]
provided $m \le \frac{n}{2}$.  Assertions (ii) and (iii) can then be proved by the same arguments as in the proof of \cite[Theorem 2.1]{BS2}. 
In this endeavor, it is useful to note that $(1+x^2)^{-m}$ is closely related to the inverse Fourier transform of the modified Bessel functions $K_\nu(|x|)$, where 
$\nu = \frac{n}{2}-m$. 
\end{proof}

In other words,  for any given $q\in\R^n, \alpha \in \R$, we have
constructed an explicit family of bounded smooth initial data (with finite mass) for which the solution of 
the free Schr\"odinger equation \eqref{eq:linSch} exhibits dispersive blow up at  the point $(x_*, t_*) = (q, \frac{1}{4\alpha})$ in space and time.   This result can be immediately generalized in 
various ways.  The  following  sequence of remarks indicates some of them.

\begin{remark}\label{amp}
The same argument shows that any initial data of the form
\[
u_0(x) = e^{- i  \alpha |x - q|^2} a(x),
\]
with an amplitude $a\in C^\infty(\R^n)\cap L^2(\R^n)\cap L^\infty(\R^n)$ but $a\not \in L^1(\R^n)$ 
will exhibit   dispersive blow up. 
Using the superposition principle,  one can construct initial data which yield 
dispersive blow up at any
countably many isolated points in space-time $\R^n \times \R$. In addition, multiplying $u_0$ by $\delta$
with  $0<\delta \ll1$, allows for initial 
data which are arbitrarily small, but which nevertheless blow up at $(x_*, t_*) $. 
By  a suitable spatial truncation of $u_0$, one can also construct small, 
smooth, bounded initial data with finite mass,  all of whose derivatives also have finite 
$L^2$--norm,  such that the corresponding solution $u$ remains smooth 
but achieves arbitrarily large values at a given  point in space-time (see \cite{BS2} for more details). 
\end{remark}

\begin{remark}\label{frac}
It was also proven in \cite{BS2} that dispersive blow up  holds true for the  class 
\begin{equation*}\label{fracS}
  \dsp i\partial_tu+(-\Delta)^{\frac{a}{2}}u=0, \quad 0<a<1,
\end{equation*}
of fractional 
Schr\" odinger equations in $\R^n\times \R^+$.
However, observe that, in contrast to the classical Schr\"{o}dinger equation, 
the phase velocity for \eqref{fracS} becomes arbitrarily large in the long wave limit but is bounded (and actually tends to zero) in the short wave limit.
\end{remark}

A natural extension of DBU for linear  Schr\"odinger equations is the
initial-value problem  in the presence of 
external potentials $V(x,t)\in \R$.  While we are not 
going to deal here with this issue in full generality, we note that an 
immediate consequence of Remark \ref{amp} is the appearance of dispersive blow up for Schr\"odinger equations with a Stark potential, {\it i.e.}
\begin{equation}\label{eq:stark}
i \partial_t v + \Delta v - (E\cdot x)v = 0, \quad\,\, v\big|_{t=0} = v_0(x),
\end{equation}
where $E\in \R^n$. This equation models electromagnetic wave propagation in a constant electric field. Solutions of \eqref{eq:stark} 
are connected to solutions of the free Schr\"odinger equation through the Avron-Herbst formula \cite{AH}. Indeed, it is easy to check that if $v$ solves 
\eqref{eq:stark}, then 
\[
u(x,t) = v\left(x +  t^2  E, t\right) e^{-i tE\cdot x -i  \frac{t^3}{3}|E|^2},
\]
solves the linear problem \eqref{eq:linSch} with the same initial data. Thus, if the  $u$ 
that solves the free Schr\"odinger equation 
exhibits dispersive blow up at a given $(x_*,t_*)$, then so does the solution $v$ of \eqref{eq:stark} at the point $(x_*+2t_*^2E,t_*)$.

An analogous result is also true in the case of linear Schr\"odinger equations with 
isotropic quadratic potentials, {\it viz.}
\begin{equation}\label{eq:linharm}
i\partial_tu + \Delta u \pm \omega^2 |x|^2 u=0, \quad u\big|_{t=0} = u_0(x),
\end{equation}
where $\omega \in \R$. The two signs correspond, respectively,
 to attractive ($-$) and repulsive ($+$) potentials. 
Following \cite{Ca}, we find that if $v$ solves \eqref{eq:linharm} with attractive harmonic potential, then 
\[
u(x,t) = \frac{1}{(1+(2\omega t)^2)^{n/4}} \, v \left(\frac{\arctan (2\omega t) }{\omega}, \frac{x}{\sqrt{1+ 
(2\omega t) ^2}}\right) e^{ i  \frac{ 2\omega^2  x^2 t}{(1+(2\omega t)^2}},
\]
solves \eqref{eq:linSch}. Thus, dispersive blow up for $u$ again implies dispersive blow up for $v$, although at a 
shifted point in  space-time.
A similar formula is available in the repulsive case (see \cite{Ca2}). 

\begin{remark}  Alternatively, one can prove dispersive blow up 
for  linear Schr\"odinger equations with quadratic potentials 
using (generalized) Mehler formulas for the associated Schr\"odinger group.
The  Mehler formulas for solutions are 
\[
u(x,t)=e^{-in\frac{\pi}{4} \text{sgn}\; t} \left|\frac{\omega}{2\pi \sin \omega t}\right|^{\frac{n}{2}}\int_{\R^n} e^{\frac{i\omega}{\sin \omega t}(\frac{|x|^2+|y|^2}{2} \cos \omega t-x\cdot y)}\, u_0(y)\, dy,
\]
 for \eqref{eq:linharm} in the attractive situation, while in the repulsive case, one has
\[
u(x,t)=e^{-in\frac{\pi}{4} \text{sgn}\; t} \left|\frac{\omega}{2\pi \sinh \omega t}\right|^{\frac{n}{2}}\int_{\R^n} e^{\frac{i\omega}{\sinh \omega t}(\frac{|x|^2+|y|^2}{2} \cosh \omega t-x\cdot y)}\, u_0(y) \, dy
\]
(see  \cite{Ca2} again).  In the attractive case, Mehler's formula is valid for $|t|<\frac{\pi}{2\omega}$, while in the repulsive case it makes sense for all $t>0$. 
Generalizations of these formulas are available and allow one to infer that dispersive blow up
can occur in the presence  of  anisotropic 
quadratic potentials (see \cite{Ca2}).
\end{remark}

We close this subsection by noting, that at least for $n=1$, it is easy to show that dispersive blow up is stable under the influence of a 
rather general class of external (real-valued) potentials $V\in C(\R; L^2(\R))$. To this end, consider the linear
Schr\"odinger equation
\begin{equation}\label{eq:linVSch}
i \partial_t u + \partial^2_{x} u - V(x,t) u=0, \quad u\big|_{t=0} = u_0(x),
\end{equation}
which which can rewritten, using Duhamel's formula, as
\begin{equation}\label{eq:Vduhamel}
u(x,t) = e^{i t \partial_x^2} u_0(x)- i \int_0^t e^{i (t-s) \partial _x^2} V(s,x) u(x,s) \, ds =: e^{i t \Delta} u_0(x) - i I_{V} (x,t) .
\end{equation}
In view of \eqref{eq:group}, we have  formally that
\begin{equation}\label{eq:I}
I_V(x,t)=\frac{1}{(4i \pi t)^{1/2}} \int_0^t\int_{\R^n}\frac{1}{(t-s)^{\frac{1}{2}}}\exp \left(i\frac{|x-y|^2}{4(t-s)}\right) V(s,y)u(y,s) \, dy \, ds.
\end{equation}
Now, assume that the first term on the right-hand side of \eqref{eq:Vduhamel} exhibits dispersive blow up at some $(x_*, t_*)$.   
Then, it suffices to show that $I_V(x,t)$ is continuous for $(x,t)\in \R^n\times [0,T]$, for some $T>t_*$ to conclude that the solution of the \eqref{eq:linVSch} 
exhibits dispersive blow up at the same point $(x_*, t_*)$.  If it is established that $I_V(x,t)$ 
is locally bounded as a function of $x$ and $t$ in the range $\R^n\times [0,T]$ for some
$T > t_*$, then Lebesgue's dominated convergence 
theorem will imply the desired continuity.  
To show local boundedness,  first apply the Cauchy-Schwartz inequality to find
\begin{equation}\label{bound}
| I_V(x,t) | \leq \frac{1}{(4\pi |t|)^{1/2}} \int_0^t \frac{1}{|t-s|^{\frac{1}{2}}} \, \| V(\cdot, s) \|_{L^2} \| \| u_0 \|_{L^2} \, ds,
\end{equation}
where conservation of mass, $\| u(\cdot, t)\|_{L^2} = \|u_0 \|_{L^2}$, has been used. 
Due to our assumption on $V$ 
and the fact that $t\mapsto t^{-1/2}$ is locally integrable, the right-hand side of \eqref{bound} is finite and we are done.
\vspace{.1cm}

A similar argument will be used preently in the study of DBU for the
 nonlinear Schr\"odinger equations in dimension $n=1$. 
It is clear, however, that in general dimensions $n>1$ a more refined analysis is needed.


\subsection{Smoothing properties of the free Schr\"odinger group}\label{sec:smooth}

In this subsection, some  technical results on the smoothing properties of the 
free Schr\"odinger group $S(t) = e^{i t \Delta}$ are reviewed. They will find use
in Section 4.   

First, recall the notion of admissible index-pairs.
\begin{definition} \label{def:adm} The pair $(p,q)$ is called {\it admissible} if
\begin{equation*}
\frac{2}{q}=\frac{n}{2}-\frac{n}{p}, \quad \text{and} \ 
    \left\lbrace
    \begin{array}{l}
    2\leq p<\frac{2n}{n-2},\ \text{for} \, n\geq 3, \\
    2\leq p<+\infty,\ \text{if} \ n=2,\\
    2\leq p\leq +\infty,\ \text{if} \ n=1.
    \end{array}\right.
\end{equation*}
\end{definition} 

From now on, for any index $r>0$,  its H\"older dual is denoted $r'$, {\it i.e.} $\frac{1}{r}+\frac{1}{r'} = 1$.
The well known Strichartz estimates for the Schr\"odinger group $S(t) =  e^{i t \Delta}$ are recounted 
in the next lemma (see \cite{Caz, LP} for more details).

\begin{lemma} \label{lem:Strich}
If $(p,q)$ is admissible, then the group $\{ e^{i t \Delta} \}_{t\in \R}$ satisfies
\[
\left(\int_{-\infty}^\infty \big\| e^{i t \Delta} f  \big\|^q_{L^p(\R^n)} \, d t\right)^{\frac{1}{q}} \leq C \big\| f \big\|_{L^2(\R^n)} 
\]
and 
\[
 \left(\int_{-\infty}^\infty \Big \| \int_{0}^t e^{i (t-s) \Delta} g(\cdot, s)  \, ds \Big \|^q_{L^p(\R^n)} \, d t\right)^{\frac{1}{q}} \! \leq 
\, C  \left( \int_{-\infty}^\infty \big\| g(\cdot, t)  \big\|^{q'}_{L^{p'}(\R^n)} \, 
d t \right)^{\frac{1}{q'}}\!, 
\]
where $C = C(p,n)$.
\end{lemma}
The estimates stated above can be interpreted as global smoothing properties of the free Schr\"odinger group $S(t)$. 
In addition to that, $S(t)$ is known to also induce local smoothing effects, some of which are collected in the following lemma (for proofs, see \cite[Chapter 4]{LP}).
For $1 \leq j \leq n$, denote the so-called homogenous derivatives of order $s>0$ by
\begin{equation}\label{eq:homD}
\begin{split}
&D^s_{x_j} f(x):= \mathcal F^{-1}  \big(|\xi_j|^s \widehat f(\xi)\big)(x) \quad \text{and, for $n=1$,} \\
&D^s f(x):= \mathcal F^{-1}  \big(|\xi|^s \widehat f(\xi)\big)(x).
 \end{split}
\end{equation}
\begin{lemma} \label{lem:smooth} 
If $n=1$ and $f \in L^2(\R)$, then
\[ 
\sup_{x\in \R} \ \int_\R  \big| D_{x}^{1/2}e^{it \partial^2_x } f (x) \big|^2 \, dt \leq C\|f\|^2_{L^2(\R)}.\]
Let $n\ge 2$. Then, for all $j\in \lbrace 1,\dots ,n\rbrace$ and $f \in L^2(\R^n)$, 
\[
\sup_{x_j\in \R} \ \int_{\R^n}  \big| D_{x_j}^{1/2}e^{it\Delta} f (x)\big|^2 dx_1\dots dx_{j-1}dx_{j+1}\dots dx_n  dt  \leq C\|f\|^2_{L^2(\R^n)}.
\]
\end{lemma}

Helpful inequalities  involving the
Schr\"odinger maximal function
\[
S_T^*f(x):= \sup_{0\leq t\leq T} |e^{it\Delta}f(x)|
\]
are derived in  \cite {RV} and  \cite{V}.  They are reported in the next lemma.
\begin{lemma}\label{lem:B}  The inequality
\begin{equation}\label{max}
\| S_T^*f\|_{L^q(\R^n)} \leq C_T \|f\|_{H^\sigma(\R^n)}
\end{equation}
holds if either
\begin{equation}\label{5.2}
n=1\quad \text{with}\quad\begin{cases}
q>2\quad\text{and}\quad \sigma \geq \max\lbrace \frac{1}{q}, \frac{1}{2}-\frac{1}{q}\rbrace, 
 \\
q=2\quad\text{and}\quad \sigma >\frac{1}{2},
\end{cases}
\end{equation}
or  
\begin{equation}\label{5.2bis}
n>1\quad \text{with}\quad\begin{cases}
q\in (2+\frac{4}{(n+1)},\infty)\quad\text{and}\quad \sigma >n(\frac{1}{2}-\frac{1}{q}),\\
q\in[2,2+\frac{4}{(n+1)}]\quad\text{and}\quad \sigma >\frac{3}{q}-\frac{1}{2}.
\end{cases}
\end{equation}
\end{lemma}

With these results at hand, attention is turned to establishing
 dispersive blow up for nonlinear Schr\"odinger equations.


\section{Dispersive blow up for nonlinear Schr\"odinger equations }\label{sec:NLS}

In this section, the initial-value problem
\begin{equation}\label{eq:NLS}
i\partial_t u+\Delta u\pm |u|^p u=0,\quad u\big|_{t=0}=u_0(x),
\end {equation}
for the nonlinear Schr\"odinger equation is considered.  Here,  $x\in \R^n$ and $p>0$ is not necessarily an integer. 
Finite-time dispersive blow up was established for $n=1$ and $p\in (0,3)$ in \cite{BS2}. 
Our strategy to improve upon this result relies upon the theory developed in  \cite{CW}, where the Cauchy problem \eqref{eq:NLS} 
was studied for $u_0\in H^s(\R^n)$ for various values of $s$.  


\subsection{Local well-posedness in $H^s$}\label{sec:LWP}

For $1\le r < \infty$ and $s>0$, define
\[
H^{s,r}(\R^n)=\Big\{ f\in L^r(\R^n): \mathcal F^{-1}\big[ (1+|\xi|^2)^{s/2} \widehat f(\xi) 
 \big] 
\in L^r(\R^n) \Big\}.
\] 
These are the  standard Bessel potential spaces. According to \cite{St}, which uses the notation $L^{s,r}$ instead of $H^{s,r}$, these spaces may be characterized in the following manner.
Let $s\in (0,1)$ and $\frac{2n}{(n+2s)}<r<\infty$. Then
$f\in H^{s,r}(\R^n)$ if and only if $f\in L^r(\R^n)$ and 
\begin{equation}\label{eq:frac}
\mathcal D^s f(x)=\left(\int_{\R^d} \frac{|f(x)-f(y)|^2}{|x-y|^{n+2s}} \, dy\right)^{1/2}\in L^r(\R^n).
\end{equation}
The space $H^{s,r}(\R^n) \equiv (I-\Delta)^{-s/2}L^r(\R^n)$ is equipped with the norm 
 \begin{align*}
 \|f\|_{H^{s,r}(\R^n)}=\|(I-\Delta)^{s/2}f\|_{L^r(\R^n)} =  & \ \|f\|_{L^r(\R^n)}+\|\mathcal D^s f\|_{L^r(\R^n)} \\
 \simeq  & \ \|f\|_{L^r(\R^n)}+\| D^s f\|_{L^r(\R^n)},
\end{align*} 
where $D^s$ is the homogenous derivative defined in \eqref{eq:homD}. 
Observe that, for $r=2$, $H^{s,2}(\R^n)\equiv H^{s}(\R^n)$, the usual $L^2$--based Hilbert space. 
In addition, a straightforward calculation reveals that 
\begin{equation}\label{3}
\|\mathcal D^s f\|_{L^2(\R^n)}=c_n\||\xi |^s \widehat f \, \|_{L^2(\R^n)} \equiv c_n \| D^s f\|_{L^2(\R^n)},
\end{equation} 
with $c_n$  a constant depending only on the dimension $n$.  
If $(q,r)$ is an admissible pair as defined in Section 2, then the space 
\begin{equation} \label{solutionspace}
W^T_{s,n}\, = \, C\big([0,T];H^s(\R^n)\big)\cap L^q\big([0,T|;H^{s,r}(\R^n)\big) 
\end{equation}
will also appear.   Of course, these spaces depend  on the admissible pair $(q,r)$, 
but this dependence is suppressed in the notation.  
The following local well-posedness theorem established in \cite{CW} makes use of both the
Bessel-potential spaces and the latter, spatial-temporal spaces.
\begin{proposition}\label{prop:CW}
Let $s>s_{p,n}=\frac{n}{2}-\frac{2}{p},$ $s>0,$ with $s$ otherwise arbitrary if $p$ is 
an even integer, $s < p+1$ if $p$ is an odd integer and  
$\lfloor s \rfloor <p$ if $p$ is not an integer. For given initial data $u_0\in H^s(\R^n)$, 
\begin{enumerate}
\item there exist $T=T(\|u_0\|_{H^s})>0$ and a unique solution
\[
u\in C\big([0,T];H^s(\R^n)\big)\cap L^q\big([0,T|;H^{s,r}(\R^n)\big)\equiv W^T_{s,n}
\]
for all  pairs $(q, r)$ admissible in the sense of Definition \ref{def:adm}, 
 and
\item the local existence time $T=T(\|u_0\|_{H^s}) \to +\infty$, as $\|u_0\|_{H^s(\R^n)}\to 0.$
\end{enumerate}
\end{proposition}

\begin{remark}The proof of this result follows  from a fixed point argument based on the Strichartz estimates displayed in Lemma \ref{lem:Strich}.  Recall that 
the notation $\lfloor s \rfloor$ connotes the largest integer less than or equal to $s$. \end{remark}

To apply Proposition \ref{prop:CW} in our context, we need to show that the class of initial data constructed in Section \ref{sec:lin} 
(yielding dispersive blow up for the free Schr\"odinger evolution) admits  Sobolev 
class regularity that will allow the use of Lemma \ref{prop:CW}. 
 Via scaling and translation, dispersive blow up can be achieved at any point $(x_*, t_*)$ in 
space-time, so without loss of generality, fix $(x_*, t_*) = (0, 1)$ and focus upon the 
initial data
\begin{equation}\label{eq:IC}
u_0(x) = \frac{e^{- 4 i |x |^2}}{(1+|x|^2)^m}\ ,\quad \text{with $\frac{n}{4} < m \le \frac{n}{2}$.}
\end{equation}
This initial value $u_0$ has Sobolev regularity explained in the next lemma.
\begin{lemma}\label{lem:linear}
Let  $u_0$ be as depicted  in \eqref{eq:IC}. Then
$u_0\in H^s(\R^n)$ if \;$2m>s+\frac{n}{2}$. In particular, if $m=\frac{n}{2},$   
then $u_0\in H^s(\R^n)$ for any $ s\in (0,\frac{n}{2})$, whereas if $m=\frac{n}{4}^+,$ then  $s \in (0,0^+)$.
\end{lemma}

\begin{proof} Consider first the case $0<s<1$.  
For $s$ in this range, Propositions 1 and 2 in \cite{NP} provide the inequalities 
\[
| \mathcal D^s e^{i|x|^2}|\leq c_n(1+|x|^s)
\]
and
\[
\|\mathcal D^s(fg)\|_{L^2(\R^n)}\leq \|f\mathcal D^s g\|_{L^2(\R^n)}+\|g\mathcal D^s f\|_{L^2(\R^n)},
\]
where $\mathcal D^s$ is defined in \eqref{eq:frac}.
Combining these estimates with identity \eqref{3} and using interpolation, one arrives 
at the inequality
\begin{align*}
&\, \left\Vert \mathcal D^s \left(\frac{e^{i|x|^2}}{(1+|x|^2)^m}\right)\right\Vert_{L^2(\R^n)} \\ 
 &\, \leq  \left\Vert\frac{1}{(1+|x|^2)^m}\mathcal D^s(e^{i|x|^2})\right\Vert _{L^2(\R^n)}+\left\Vert \mathcal D^s\left(\frac{1}{(1+|x|^2)^m}\right)\right\Vert_{L^2(\R^n)}\\
&\, \leq c_n \left\Vert \frac{1}{(1+|x|^2)^m}\right\Vert_{L^2(\R^n)}+c_n \left\Vert \frac{|x|^s}{(1+|x|^2)^m}\right\Vert_{L^2(\R^n)}\\
&\quad \  +\left\Vert \frac{1}{(1+|x|^2)^m}\right\Vert_{L^2(\R^n)}^{1-s} \left\Vert D\left(\frac{1}{(1+|x|^2)^m}\right)\right\Vert^s_{L^2(\R^n)}.
 \end{align*}
The right-hand side of this inequality is finite if and only if $2m-s>\frac{n}{2}.$  

If $s$ is a positive integer, the result follows from Leibnitz' rule. 
If instead, $s = s' +k$ where $0 < s' < 1$, then simply apply the above analysis 
to the $k^{\rm th}$--deriviative $u_0^{(k)}$.    
\end{proof}

Notice that if $m = \frac{n}{2}$, we can certainly choose the value $s$ in the interval $(\frac{n}{2}-\frac{2}{p},\frac{n}{2}]$ where Proposition  \ref{prop:CW}  applies.


\subsection{Proof of dispersive blow up for nonlinear Schr\"odinger equations}\label{sec:DBUNLS}

Here is the detailed statement of dispersive blow up for the initial-value problem \eqref{eq:NLS}.
In this theorem, the 
Lebesgue index $p\geq \lfloor \frac{n}{2} \rfloor$ if $p$ is not an even integer.

\begin{theorem}\label{thm:main}
Given $t_{*}>0$ and $x_*\in \R^n,$ for any $s\in (\frac{n}{2}-\frac{1}{2p},\frac{n}{2}]$,
 there are initial data  $u_0\in H^s(\R^n)\cap L^{\infty}(\R^n)\cap C^{\infty}(\R^n)$ 
and a $T = T(\|u_0\|_{H^s})>t_*$
 such that 
\begin{enumerate}   
\item the  initial-value problem   \eqref{eq:NLS} has a unique solution $u$ 
in the class described in Proposition   \ref{prop:CW} which is defined 
at least on  the time interval $[0,T]$,   and
\item $u$ exhibits dispersive blow up, which is to say,
$$ \lim_{{(x,t)\in \R^n\times [0,T] \to (x_*,t_*)}} |u(x,t)|=+\infty.$$
\item Moreover, $u$ is a continuous function of $(x,t)$ on $\R^n\times( [0,T])\setminus \lbrace t_*\rbrace)$ and
\item $u(\cdot,t_*)$ is a continuous function of $x$ on $\R^n\setminus \lbrace x_* \rbrace.$
\end{enumerate}
\end{theorem}

This theorem extends the results of \cite{BS2} to the cases where $n\geq 2$ and $p\geq 3.$ 
Notice that the nonlinearity $y \mapsto |y|^py$ is smooth when $p$ is an even integer.  
Otherwise, it has finite regularity and hence the restriction on $p$ in those cases.

\begin{proof} The proof is provided in detail  for $p>0$ in the case $n=1$ and, when $n\geq 2$, 
for the case $p=2k$,
$k$ a positive integer.  It 
will be clear from the argument that the result extends to the case of 
 $p\geq \lfloor \frac{n}{2} \rfloor$ if $p$ is not an even integer. 

As already mentioned, we may assume that the  dispersive blow up for the 
free Schr\"odinger group $S(t)$ 
occurs at $x_* = 0$ and $t_*=1$. Note that the same is true for initial data of the form $\delta u_0$, where $u_0$ is 
as in \eqref{eq:IC} with $m = \frac{n}{2}$, say, $\delta > 0$  arbitrary and $s$ satisfying 
the conditions in Proposition \ref{prop:CW}. 
In view of  part (3) of the latter proposition,  the local existence time $T^* = T(\| \delta u_0\|_{H^s})>0$ can be made arbitrarily large by choosing $\delta$  sufficiently small 
and hence we can always achieve $T^*>1 = t_* $.\\

{\it Step 1.}  Take as initial 
data $\delta u_0$, where $u_0$ is as in \eqref{eq:IC} with $m = \frac{n}{2}$.  Let $s \in (\frac{n}{2} - \frac{2}{p}, \frac{n}{2}]$ with $p \geq \lfloor\frac{n}{2}\rfloor$.   Then 
$s$ satisfies the conditions of Proposition \ref{prop:CW}.  As noted above, by choosing 
$\delta$ small enough, we can be sure that the solution $u$ of \eqref{eq:NLS} 
emanating from $u_0$ 
exists and is unique in $C([0,T]:H^s)$ where $T > t_* = 1$.   

Duhamel's formula allows us to represent $u$ in the form
\begin{equation}\label{eq:duhamel}
u(x,t) = e^{i t \Delta} u_0(x) \pm i \int_0^t e^{i (t-s) \Delta} |u(x,s)|^p u(x,s) \, ds =: e^{i t \Delta} u_0(x) \pm i I (x,t) ,
\end{equation}
where, at least formally, $I(x,t)$ can be written as the double integral
\begin{equation}\label{eq:I}
I(x,t)=\frac{1}{(4i \pi t)^{n/2}} \int_0^t\int_{\R^n}\frac{1}{(t-s)^{\frac{n}{2}}}\exp \left(i\frac{|x-y|^2}{4(t-s)}\right) |u(y,s)|^p  u(y,s) \, dy \, ds.
\end{equation}
The first term on the right-hand side of \eqref{eq:duhamel} 
 exhibits dispersive blow up at $(x_*, t_*) = (0,1)$ on account of the choice of $u_0$.   
 If it turns out that  $I$ is continuous for $(x,t)\in \R^n\times [0,T]$, then it is 
immediately concluded that \eqref{eq:duhamel} (and thus \eqref{eq:NLS}) exhibits 
dispersive blow up at the same point $(x_*, t_*)=(0,1)$.  To show that $I(x,t)$ 
is continuous, it suffices to prove that it is locally bounded as a function of $x$ and $t$, 
since then  Lebesgue's dominated convergence theorem will imply $I$ is continuous on $\R^n\times [0,T]$. \\

{\it Step 2.}  Consider  $n=1$ first, since  a more direct proof can be made 
in this case.  
The  initial data is 
\[
u_0(x)= \, \frac{\delta e^{-i4 x^2}}{(1+x^2)^{\frac{1}{2}}},\quad x\in \R.
\] 
The function $u_0$ lies in $H^s(\R)$ for any $s$ in the range  $0\leq s<\frac{1}{2}.$ 
Proposition \ref{prop:CW} then provides a local in time solution $u \in C([0,T]; H^s(\R))$ to the Cauchy problem \eqref{eq:NLS} provided that 
\[
0<s<\frac{1}{2} \quad \,\, \text{and} \,\, \quad 0<p\leq \frac{4}{1-2s}.
\]
As mentioned already, $T$ may be taken larger than 1 by choosing $\delta$ small.  By Sobolev imbedding, $H^s(\R)\subset L^{r+1}(\R)$ if $\frac{1}{r+1}\geq \frac{1}{2}-s.$
Hence, for  $s=\frac{1}{2}-\eps$, where $\eps>0$ is fixed and small, it is inferred  
that $u\in C([0,T]; L^{r+1}(\R))$ for all $r$ in the range 
\[
0<r\le \min \left\{\frac{2}{\eps}, \frac{1-\eps}{\eps}\right\}=\frac{1-\eps}{\eps}. 
\]
As $\epsilon > 0$ was arbitrary, it  follows that  
$u(\cdot, t)\in L^{r+1}(\R)$, for  $r\geq 1$ arbitrarily large. 
In consequence, $I(x,t)$ is locally bounded.  
Indeed, using H\"older's inequality, it is seen that
\begin{align*}
|I(x,t)| \le &\ \int_0^t \frac{1}{(t-s)^{1/2}} \, \big\| u(\cdot, s)\big\|_{L^{p+1}}^{p+1} \, ds \\
\le &\  \left(\int_0^t \frac{1}{(t-s)^{\gamma/2}} \, ds \right)^{\frac{1}{\gamma} }
 \left(\int_0^t \big\| u(\cdot, s)\big\|_{L^{p+1}}^{\gamma' (p+1)} \, ds \right)^{\frac{1}{\gamma'} }
\end{align*}
where $\frac{1}{\gamma}+\frac{1}{\gamma'}=1$ and $\gamma \in (1,2)$. Having in mind that $u\in C([0,T]; L^{p+1}(\R))$, it transpires that
\[ |I(x,t)| \le C T^{\frac{1}{\gamma'}} \sup_{t\in[0,T]} \| u(\cdot, t) \|_{L^{p+1}}^{\gamma'(p+1)},
\]
for all $x \in \R$, which concludes the proof in the case $n=1$.\\

{\it Step 3.} In  case $n\ge 2$,  the strategy employed 
above no longer works beause the factor $t\mapsto t^{-n/2}$ appearing 
in the representation formula \eqref{eq:I} is no longer locally integrable. However, it will be shown 
in Proposition \ref{prop:key} below that the double integral $I$ is in fact half a derivative 
smoother than one would naively expect. To make use of this result,  choose initial data $u_0$ 
of the form \eqref{eq:IC} with $m \in (\frac{n}{2}-\frac{1}{4p},\frac{n}{2}]$, where 
$p\ge \lfloor\frac{n}{2}\rfloor$.
It is immediately inferred from Lemma \ref{lem:linear} that 
\[u_0\in H^s(\R^n)\quad \text{for any} \,\, s\in \Big( \frac{n}{2}-\frac{1}{2p},\frac{n}{2}\Big].\]
Notice that 
\[\frac{n}{2}-\frac{1}{2p}>\frac{n}{2}-\frac{2}{p}=s_{p,n}\] 
so that Proposition \ref{prop:CW} applies and therefore $u\in W^T_{s,n}$ (see \eqref{solutionspace})
satisfies the integral form \eqref{eq:duhamel} of the nonlinear Schr\"odinger equation.  
Proposition \ref{prop:key} below then shows  that 
\[
I \in C([0,T]; H^{s+1/2}(\R^n)),
\]
and hence, for $s\in (\frac{n}{2}-\frac{1}{2p},\frac{n}{2}],$ one has
\[
I \in C(\R^n\times [0,T])\cap L^\infty(\R^n\times [0,T]).
\]
The proof is complete.
\end{proof}

\begin{remark} It was shown in \cite[Theorem 2.1]{BS2} that dispersive blow up results 
are stable under smooth and localized perturbations of the data.  That is to say,  
if they hold for data $u_0,$ then they also hold for data $u_0+w$ where, for instance, $w \in H^\infty(\R^n).$   In particular the data 
leading to DBU do {\it not} need to be radially symmetric. The same is true for Theorem \ref{thm:main}.
The proofs of these results consists of writing the equation satisfied by $w$ and showing
that it has bounded, continuous solutions.  The details follow exactly the argument 
given already in \cite{BS2}.  

In addition, there is  a kind of density of initial data leading to dispersive blow up.  More 
precisely,  given $u_0\in H^s(\R^n)$ with $s>n/2$ and $\epsilon >0,$ there exists
$\phi\in H^r(\R^n), $ $r\in \big(\frac{n}{2}-\frac{1}{2p},\frac{n}{2}\big]$ with 
\begin{equation}
  \label{density}
    \|u_0-\phi \|_{H^r(\R^n)}<\epsilon,
\end{equation}
such that the initial data $\phi$ leads to dispersive blow up in the sense of Theorem \ref{thm:main}. Indeed, it suffices to take 
$\phi=u_0+\delta v_0,$ where $v_0$ leads to dispersive blow up and  $\delta >0$ is small enough 
that \eqref{density} holds. A combination of Theorem \ref{thm:main} and 
 Proposition \ref{prop:CW} then 
implies the above assertion.
\end{remark}


\section{Global smoothing of the Duhamel term and applications}\label{sec:global}

In this section, the proof of Theorem \ref{thm:main} is completed by showing the Duhamel 
term in the integral representation \eqref{eq:duhamel} of the solution  of the initial-value problem \eqref{eq:NLS}
is smoother than is the linear term involving only the initial data.  In fact, several
different results of smoothing by the Duhamel term will be developed, though 
the first one is enough for the dispersive blow up result in Section 3.  

\subsection{Smoothing by half a derivative}

The following proposition
suffices to complete the proof 
of Theorem \ref{thm:main}.  

\begin{proposition}\label{prop:key}
Let $u_0\in H^s(\R^n),$\;$s>\frac{n}{2}-\frac{1}{2p}$ with $p\ge 1$ 
and $\lfloor p+1\rfloor \geq s + \frac12$ if $p$ is not an even integer. 
Let  $u\in W^T_{s,n}$ be the solution of \eqref{eq:NLS} satisfying
\[
u(x,t)=e^{it\Delta}u_0(x)\pm i \int_0^te^{i(t-s)\Delta} |u(x,s)|^pu(x,s)ds=:e^{it\Delta}u_0(x) \pm i I(x,t).
\]
Then $I \in C([0,T]; H^{s+\frac12}(\R^n)).$ 
\end{proposition}
In other words, the integral term $I$ is \lq\lq smoother" than the free propagator $e^{it\Delta}u_0$ by 
half a derivative.
This is  the key point needed in the proof of Theorem \ref{thm:main} for $n\ge 2$.   In the 
special case wherein
the nonlinearity $|u|^pu$ is smooth, so when $p = 2k, k$ an integer,  
Proposition \ref{prop:keybis} will show that  the Duhamel term is
almost one derivative smoother than one would expect.

\begin{remark}  The fact that
the nonlinear integral term in  Duhamel's formula is smoother than the linear one
in certain circumstances has been used in other  works 
on nonlinear dispersive equations. 
For example, in \cite{LS} it was employed to give a different 
proof of some of the results obtained in \cite{BS1}. 
In \cite{Bou}, this smoothing effect was applied to deduce  global 
well-posedness below the regularity index provided by the conservation laws of 
mass and energy.  So far as we are aware, however, the result stated in Proposition 
\ref{prop:key} has not
previously been explicitly written down.
 \end{remark}

\begin{proof} The details are provided for the case $p=2k, k\in \N.$ 
It will be clear that the arguments extend to the case 
where $p\geq 1$ is not an even integer, but $p\ge \lfloor\frac{n}{2}\rfloor$ and $n\ge2$. \\

{\it Step 1.} In the first step, useful estimates on $e^{it\Delta}u_0$ are derived. 
Start by fixing a $j\in \lbrace 1,\dots,n\rbrace$ and noticing that
\begin{equation}\label{compA}
\sup_{0\leq t\leq T}\left\{ \sup_{x_1\dots x_{j-1}x_{j+1}\dots x_n}\Big\{|e^{it\Delta} u_0(x)|\Big\} \right\} \lesssim 
\sup_{x_1\dots x_{j-1}x_{j+1}\dots x_n}\left\{|e^{it(x_j)\Delta}u_0(x)|\right\}
\end{equation}
for some $t(x_j)\in [0,T]$.  For  $q\ge 2$ and $s>\frac{(n-1)}{q}$, it follows by Sobolev embedding that
\begin{equation}\label{comp}
\begin{aligned}
\sup_{x_1\dots x_{j-1}x_{j+1}\dots x_n}|e^{it(x_j)\Delta}u_0(x)| & \, \lesssim  
\big\|e^{it(x_j)\Delta}u_0\big\|_{H^{s,q}_{x_1\dots x_{j-1}x_{j+1}\dots x_n}\left(\R^{n-1}\right)} \\
&\, =c\big\|e^{it(x_j)\Delta}J^s_{\hat \jmath}u_0\big\|_{L^q_{x_1\dots x_{j-1}x_{j+1}\dots x_n}\big(\R^{n-1}\big)}\\
&\, \lesssim \big\|\sup_t |e^{it\Delta}J^s_{\hat \jmath}u_0|\;
\big\|_{L^q_{x_1\dots x_{j-1}x_{j+1}\dots x_n} \left(\R^{n-1}\right)}
\end{aligned}
\end{equation}
where here and in the following, 
\[
J^s_{\hat \jmath }=(1-(\partial^2_{x_1}+\dots+\partial^2_{x_{j-1}}
+\partial^2_{x_{j+1}}+\dots+\partial^2_{x_n}))^{s/2}
\]
is defined via the associated Fourier symbol and the subscripts on the function spaces 
indicate which variables participate in the norm. 
Using \eqref{compA} and \eqref{comp} together with 
the estimates on the Schr\"odinger maximal function given in Lemma \ref{lem:B}, 
there appears the inequality
\begin{equation}\label{ouf}
\begin{aligned}
&\big\|e^{it\Delta}u_0\big\|_{L^q_{x_j}(\R;L^\infty_{x_1\dots x_{j-1}x_{j+1}\dots x_n t}(\R^{n}\times[0,T]))}\\
&\hspace{1cm}= \Big\|\sup_{0\leq t\leq T}\sup_{x_1\dots x_{j-1}x_{j+1}\dots x_n}|e^{it\Delta}u_0|\;\Big\|_{L^q_{x_j}(\R)}\\
&\hspace{1cm} \leq C_T\Big\|\sup_{0\leq t\leq T}|e^{it\Delta}J^s_{\hat \jmath}u_0|\;\Big\|_{L^q(\R^n)}\\
&\hspace{1cm} \lesssim \big\|u_0\big\|_{H^{\sigma+s}(\R^n)},
\end{aligned}
\end{equation}
with $q\geq 2$, $s>\frac{n-1}{q}$ and $\sigma>0$ as specified in Lemma \ref{lem:B}. 
The inequality \eqref{ouf}, together with the local smoothing estimates stated in Lemma \ref{lem:smooth}, 
reveal that
\begin{equation}\label{ouf2}
\big\|D^{\theta/2}_{x_j}e^{it\Delta} u_0\big\|_{L_{x_j}^{q/(1-\theta)}(\R;L^{2/\theta}_{x_1\cdot\cdot x_{j-1}x_{j+1}\cdot\cdot x_nt}(\R^{n-1}\times[0,T]))} 
\leq c_T\big\|u_0\big\|_{H^{(1-\theta)(\sigma+s)}(\R^n)},
\end{equation}
where $\theta\in[0,1]$, $q\geq 2$, $s>\frac{n-1}{q}$ and $\sigma>0$ as before. \\

{\it Step 2.}  To bound $\|D_{x_j}^{s+1/2} I\|_{L^\infty_t([0,T]:L^2(\R^n))}$ above,  let  $j\in \lbrace 1,\dots , n\rbrace$, $p=2k$ and write
\begin{equation}
\label{tterm}
\begin{aligned}
&\Big \|D_{x_j}^{s+1/2}  \int_0^t e^{i(t-s)\Delta}(|u|^{2k}u)(s) \, ds \Big \|_{L^\infty([0,T];L^2(\R^n))}\\
&\hspace{1cm} \leq \, c\big\|D_{x_j}^{s+1/2} (|u|^{2k} u)\big\|_{L^1([0,T];L^2(\R^n))}\\
&\hspace{1cm} \leq \, cT^{1/2}\big\|D_{x_j}^{s+1/2} (|u|^{2k} u)\big\|_{L^2([0,T]\times \R^n)}.
\end{aligned}
\end{equation}
 To estimate the right-hand side of \eqref{tterm}, the calculus of inequalities involving fractional derivatives 
derived  in \cite{KPV93} is helpful.  More precisely,  the following inequality, 
which is a particular case of those proved in 
 \cite[Theorem A.8]{KPV93}, will be used.   Let $\alpha\in (0,1), \,\alpha_1,\alpha_2\in[0,\alpha]$ with $\alpha=\alpha_1+\alpha_2$ and  let $p_1,p_2,q_1,q_2\in [2,\infty)$ be such that
 \[\frac{1}{2}=\frac{1}{p_1}+\frac{1}{p_2}=\frac{1}{q_1}+\frac{1}{q_2}.\] Then
 \begin{equation}
 \label{inequality}
 \begin{aligned}
 &\big\| D_{x_j}^{\alpha}(fg)-fD_{x_j}^{\alpha}f-g D_{x_j}^{\alpha}f  \big\|_{L^2_{x_j}(\R;L^2(Q))}\\
& \hspace{.7cm}\leq c\big\| D_{x_j}^{\alpha_1}f\big\|_{L^{p_1}_{x_j}(\R;L^{q_1}(Q))}   \big\|D_{x_j}^{\alpha_2}g\big\|_{L^{p_2}_{x_j}(\R;L^{q_2}(Q))},
 \end{aligned}
 \end{equation}
 where $Q=\R^{n-1}\times [0,T]$. 

To illustrate the use of the inequality \eqref{inequality} in estimating the right-hand side of \eqref{tterm}, 
assume without loss of generality that $s+\frac{1}{2}=1+\alpha$ with  $\alpha\in (0,1)$. 
Thus, omitting the domains of integration $\R$ and $Q$,
\begin{equation}
\label{001}
\begin{aligned}
\big\| D_{x_j}^{s+1/2}(fg)\big\|_{L^2_{x_j}L^2}& = \big\| D_{x_j}^{1+\alpha}(fg)\big\|_{L^2_{x_j}L^2} \simeq \big\| D_{x_j}^{\alpha}\partial_{x_j}(fg) \big\|_{L^2_{x_j}L^2}\\
&\leq c\Big( \big\| D_{x_j}^{\alpha}(f\partial_{x_j}g)\big\|_{L^2_{x_j}L^2}
\big\| D_{x_j}^{\alpha}(\partial_{x_j}f g)\big\|_{L^2_{x_j}L^2}\Big).
\end{aligned}
\end{equation}
By symmetry it suffices to consider only one of the terms 
on the right-hand side of \eqref{001}. From  \eqref{inequality}, with $\alpha_1=0$, 
there obtains
\begin{equation*}
\label{002}
\begin{aligned}
\big\| D_{x_j}^{\alpha}(g\partial_{x_j}f)\big\|_{L^2_{x_j}L^2}& \leq  \big\| D_{x_j}^{\alpha}(g\partial_{x_j}f)-g D_{x_j}^{\alpha}\partial_{x_j}f - \partial_{x_j}fD_{x_j}^{\alpha}g \big\|_{L^2_{x_j}L^2}\\
& \quad +\big\| g D_{x_j}^{\alpha}\partial_{x_j}f\big\| _{L^2_{x_j}L^2} + \big\|  \partial_{x_j}f D_{x_j}^{\alpha}g\big\| _{L^2_{x_j}L^2}\\
&\leq \big\| g D_{x_j}^{\alpha}\partial_{x_j}f \big\| _{L^2_{x_j}L^2} + c\big\| \partial_{x_j}f\big\|_{L^{p_1}_{x_j}L^{q_1}}   \big\|D_{x_j}^{\alpha_2}g\big\|_{L^{p_2}_{x_j}L^{q_2}}
\end{aligned}
\end{equation*}
with $p_1,p_2,q_1, q_2$ restricted as above.

Using the latter inequality to continue the inequality \eqref{tterm} yields
\begin{equation}\label{KPV}
\begin{aligned}
&\big\|D_{x_j}^{s+1/2} I\big\|_{L^\infty_t([0,T]:L^2(\R^n))}  \\ 
 &\hspace{.8cm} \leq  \ cT^{1/2} \big\|D_{x_j}^{s+1/2} (|u|^{2k} u)\big\|_{L^2([0,T]\times \R^n)}\\
&\hspace{.8cm} \leq \ cT^{1/2}\Big ( \big\|u\big\|^{2k}_{L^{4k}_{x_j}(\R; L^\infty_{x_1\dots x_{j-1}x_{j+1}\dots x_n,t}(\R^{n-1}\times[0,T]))} \\
& \hspace{1.5cm} \times \big\|D_{x_j}^{s+1/2} u\big\|_{L^{\infty}_{x_j}(\R; L^2_{x_1\dots x_{j-1}x_{j+1}\dots x_n,t}(\R^{n-1}\times[0,T]))}+R\Big),
\end{aligned}
\end{equation}
where the remainder $R$ includes only estimates for terms involving powers of $u$,  $\partial_{x_j}u$ and $D_{x_j}^{\alpha}u$ . These are straightforwardly bounded above 
 by use of \eqref{ouf2}. In fact, a bound for them is an interpolation between the first two terms 
on the right-hand side of \eqref{KPV}. It therefore remains to bound only the terms
\begin{equation}\label{eq:A}
 \big\|u\big\|^{2k}_{L^{4k}_{x_j}(\R; L^\infty_{x_1\dots x_{j-1}x_{j+1}\dots x_n,t}(\R^{n-1}\times[0,T]))}
\end{equation}
and
\begin{equation}\label{eq:B}
\big\|D_{x_j}^{s+1/2} u\big\|_{L^{\infty}_{x_j}(\R; L^2_{x_1\dots x_{j-1}x_{j+1}\dots x_n,t}(\R^{n-1}\times[0,T]))},
\end{equation}
$j = 1, \cdots ,n$, appearing in \eqref{KPV}.\\

{\it Step 3.} To bound the quantity appearing in  \eqref{eq:A}, first note that \eqref{ouf} implies
\begin{equation*}
\big \|\sup_{0\leq t\leq T}\quad\sup_{x_1\dots x_{j-1}x_{j+1}\dots x_n }|e^{it\Delta}u_0|\;\big \|_{L^{4k}_{x_j}(\R)}
\lesssim \big\|\sup_{0\leq t\leq T} |e^{it\Delta} J^s_{\hat \jmath} u_0|\;\big\|_{L^{4k}(\R^n)},
\end{equation*}
with $s>\frac{(n-1)}{4k}$. This estimate can be extended using Lemma \ref{lem:B} by 
observing that 
\begin{equation*}\label{une}
\begin{aligned}
\Big\|\sup_{0\leq t\leq T} |e^{it\Delta} J^s_{\hat \jmath} u_0|\;\Big\|_{L^{4k}(\R^n)} 
\leq \big\| J^s_{\hat \jmath} u_0\big\|_{H^{\sigma}(\R^n)}\lesssim \big\|u_0\big\|_{H^{\sigma+s}(\R^n)}=\big\|u_0\big\|_{H^{\bar{s}}(\R^n)},
\end{aligned}
\end{equation*}
where
\[
\bar{s} =s+\sigma >\frac{n-1}{4k}+n\Big(\frac{1}{2}-\frac{1}{4k}\Big)=\frac{n}{2}-\frac{1}{4k}.
\]
Inserting this inequality in the Duhamel representation \eqref{eq:duhamel} with 
\[
s\in\Big(\frac{n}{2}-\frac{1}{4k},\frac{n}{2}\Big]=\Big(\frac{n}{2}-\frac{1}{2p},\frac{n}{2}\Big],
\]
it follows that
\begin{equation*}
\big\|u\big\|_{L^{4k}_{x_j}(\R;L^\infty_{x_1\cdot\cdot x_{j-1}x_{j+1}\cdot\cdot x_n t}(\R^{n-1}\times[0,T]))}\leq C\big\|u_0\big\|_{s,2}+
\big\|J^s(|u|^pu)\big\|_{L^1_t([0,T]; L^2(\R^n))}.
\end{equation*}
Since $p=2k$, use of  a fractional Leibniz rule (see \cite{KP}) implies that
\begin{equation}\label{KaPo}
\big\|J^s(|u|^{2k}u)\big\|_{L^1([0,T]; L^2(\R^n))}\leq c\big\|u\big\|^{2k}_{L^{2k}([0,T];L^\infty(\R^n))}\big\|J^su\big\|_{L^\infty([0,T];L^2(\R^n))}.
\end{equation}

If it was known that 
\begin{equation}\label{enfin}
\|u\|_{L^{2k}([0,T];L^\infty(\R^n)))}\leq cT^{\theta}\|J^su\|_{L^q([0,T];L^r(\R^n))}
\end{equation}
for some $\theta>0$ and for some admissible Strichartz pair $(r,q)$, then the sequence 
of inequalities could be closed. To obtain \eqref{enfin},  recall that 
\[
s>\frac{n}{2}-\frac{1}{4k}=\frac{2kn-1}{4k}=\frac{n}{4kn/(2kn-1)},
\]
so we can 
take 
\[
r=\frac{4kn}{(2kn-1)}<\frac{2n}{(n-2)}, \quad \text{if $n\geq 3$}, 
\]
(the cases $n=1$ or $2$ are immediate) and $q=8k.$ By Sobolev embedding, the inequality  \eqref{enfin} 
is seen to hold with $\theta=\frac{3}{8k}$. In summary, all the terms in  
  \eqref{eq:A} are shown to be bounded.\\

{\it Step 4.} Finally, attention is turned to terms of the form appearing in \eqref{eq:B}. 
 The local smoothing estimate enunciated in Lemma \ref{lem:smooth} together 
with  Duhamel's formula imply that  
\begin{eqnarray*}
&\big\|D_{x_j}^{s+1/2} u\big\|_{L^{\infty}_{x_j}(\R; L^2_{x_1\cdot \cdot x_{j-1}x_{j+1}\cdot \cdot x_n,t}(\R^{n-1}\times[0,T]))}
\\
&\hspace{.8cm}   \leq \big\|u_0\big\|_{H^{s}(\R^n)}+\big\|J^s(|u|^{2k}u)\big\|_{L^1([0,T];L^2(\R^n))}.
\end{eqnarray*}
The right-hand side was already estimated in \eqref{KaPo}--\eqref{enfin}. 
Because of \eqref{KPV}, this shows that  there exists a $C=C(T,n)>0$ such that
\[
\big\|D_{x_j}^{s+1/2} I\big\|_{L^\infty([0,T]:L^2(\R^n))} \leq  C.
\]
Summing these estimates over $j$ for  $j=1\cdots,n$  yields the result advertised in 
the proposition.   

Finally, we remark that in the case where  $p$ is not an even integer, the restriction 
on $p$ are necessary and one needs 
to supplement  the Leibnitz-type inequality \eqref{inequality} with
the chain rule for fractional derivatives adduced in the Appendix of \cite{KPV93}. 
\end{proof}


\subsection{An even stronger smoothing property}

For large $s$ and higher values of $p$, a stronger smoothing 
 result than  that established in Proposition \ref{prop:key} holds.
\begin{proposition}\label{prop:keybis}
Let $u_0\in H^s(\R^n), s > \frac{n}{2} - \frac{1}{2p}$ with $p\geq 2$ and 
$\lfloor p+1\rfloor \geq s + \frac{1}{2}$ if $p$ is not an even integer.  
Under these hypothses, it follows that for any $\eps > 0$,  
\[
I\in C([0,T];H^{s+1-\eps}(\R^n)),
\]
 where the  notation is taken from Proposition \ref{prop:key} 
\end{proposition}
\begin{remark}The loss of $\eps$ in the regularity of $I$ is needed to obtain a factor 
$T^{\delta(\eps)},  \delta(\eps)>0$, on the right-hand side of the inequalities below which 
allows them to be closed. 
It can be recovered by assuming that the data $u_0$ is small enough in $H^s(\R^n).$
\end{remark}
The proof of Proposition {\ref{prop:keybis} uses the following smoothing estimate, which is a direct consequence of 
Lemma \ref{lem:smooth}, a duality argument and the  Christ-Kiselev lemma \cite{CK}. For a proof, see \cite[Chapter 4]{LP}.

\begin{lemma}\label{lem:Kis} For any $n\in \N$, the inequality
\[
\left\|D^{1/2}_{x_j}\int_0^te^{i(t-s)\Delta}f(\cdot,s) \, ds\right\|_{L^\infty(\R;L^2(\R^n))}
\leq C\big\|\mathcal H_j f \big\|_{L^1_{x_j}(\R;L^2_{x_1...x_{j-1}x_{j+1}...x_n t}(\R^n))}
\]
holds, where $\mathcal H_j$ denotes the {\it Hilbert transform} in the $j$-th variable, which 
is to say, 
\[
\mathcal H_j f(x):=-i \, \mathcal F^{-1} \Big ( \text{{\rm sign}}(\xi_j) \widehat f(\xi) \Big)(x) .
\]
\end{lemma}

\begin{proof}[Proof of Proposition \ref{prop:keybis}] 
The proof is similar to that of Proposition \ref{prop:key} and hence, we 
only sketch the main differences. 

First consider the case of  data $u_0\in H^s(\R^n)$
which is small, so that all the norms involved are indeed \lq\lq small". 
We want to show that the integral term $I$ in Duhamel's formula is one order smoother in the Sobolev scale $C([0,T];H^s(\R^n))$ 
than the the free propagation $e^{it\Delta} u_0.$ To this end, apply Lemma \ref{lem:Kis} together with the commutator estimate in 
\cite[Theorem A.13]{KPV93} to write
\begin{equation}\label{eq:L3}
\begin{aligned}
&\ \big\|\mathcal H_j D^{s+1}_{x_j} I \big\|_{L^\infty([0,T];L^2(\R^n))} \\ 
&\hspace{.8cm} \lesssim 
\big\| D^{s+1/2}_{x_j}(|u|^{2k}u) \big\|_{L^1_{x_j}(\R;L^2_{x_1\cdot\cdot x_{j-1}x_{j+1}\cdot\cdot x_n t}(\R^{n-1}\times[0,T] ))}\\ 
& \hspace{.8cm} \lesssim\Big (\|u\big\|^{2k}_{L^{2k}_{x_j}(\R;L^2_{x_1\cdot\cdot x_{j-1}x_{j+1}\cdot\cdot x_n t}(\R^{n-1}\times[0,T] ))} \\
&\hspace{1.3cm} \times \|D_{x_j}^{s+1/2}u\|_{L^{\infty}_{x_j}(\R;L^2_{x_1\cdot\cdot x_{j-1}x_{j+1}\cdot\cdot x_n t}(\R^{n-1}\times[0,T] ))}+R\Big).
\end{aligned}
\end{equation}
To estimate the  two explicit quantities on the right-hand side of the last inequality,  
one uses arguments similar to those given in the proof of Proposition \ref{prop:key}. 
The estimates for the remainder terms represented by $R$ then follow by interpolation 
of the previous estimates. Since the terms on the right-hand side of \eqref{eq:L3} are 
quadratic and each factor is small, one can close the estimate and get the desired result, but only provided that $u_0$ is sufficiently small.

For data $u_0\in H^s(\R^n)$ of arbitrary size one gives up $\eps$-amount of spatial
smoothing for a little temporal smoothing, thereby obtaining the 
 factor  $T^{\delta(\epsilon)}$, $\delta(\eps)>0$  on the  right-hand side.  The 
right-hand side of the estimate then has lower homogeneity than the left side 
and the proof proceeds.
\end{proof}


\subsection{Extension to the case of non-elliptic Schr\"odinger equations} The results above extend to the case of non-elliptic, non-degenerate, nonlinear Schr\"odinger equations of the form
\begin{equation}\label{eq:hypNLS}
i\partial_t u+\Delta_{\rm H} u\pm |u|^p u=0,\quad u\big|_{t=0}=u_0(x),
\end{equation}
where 
\[
\Delta_{\rm H} := \partial^2_{x_1} + \dots  \partial^2_{x_j} -  \partial^2_{x_{j+1}}\dots -  \partial^2_{x_n}.
\]
\begin{proposition}\label{prop:hypNLS}
The result of Theorem \ref{thm:main} also holds  for the initial-value 
problem delineated in  \eqref{eq:hypNLS} .
\end{proposition}
\begin{proof} Remark first that for initial data  of the form 
\begin{equation}
\label{data-hyp}
\widetilde u_0(x)=
\frac{e^{-i\alpha\big((x_1-q_1)^2+\dots+(x_j-q_j)^2-(x_{j+1}-q_{j+1})^2-\dots-(x_n-q_n)^2\big)}}{(1+|x|^2)^m},
\end{equation}
with $n/4<m\leq n/2$, the solution of the  initial-value problem associated to 
the linearization of \eqref{eq:hypNLS} around the rest state, 
i.e., \[u(x,t) =e^{it\Delta_H}\widetilde u_0(x),\] 
satisfies the conclusions of Lemma \ref{lem:IC}, and in particular, blows up 
 at the point $ \widetilde q=(q_1,\dots,q_j,-q_{j+1},\dots,-q_n)$.

Next,  notice  that the the local well-posedness result stated in Proposition \ref{prop:CW}, is solely 
based on Strichartz estimates, which are exactly the same for the two groups $\{e^{it \Delta} \}_{t\in \R}$ and 
$\{e^{it {\Delta_{\rm H}}} \}_{t\in \R}$ ({\it cf.} \cite{GS}). 
In other words, Lemma \ref{lem:Strich} also holds in the  non-elliptic case.
 Together with Sobolev embeddings, this yields a unique solution $u\in W_{s,n}^{T}$ to \eqref{eq:hypNLS}, 
by the same arguments as in \cite{CW}. Furthermore, the local smoothing estimates stated in Lemma \ref{lem:smooth} 
carry over to the  non-elliptic situation.   

Finally, for the boundedness of the associated maximal function 
(see Lemma \ref{lem:B}), we point to the 
inequality
\begin{equation}
\label{max-non-ellip}
\Big \| \sup_{0\leq t\leq T} \big |e^{it(\partial^2_{x_1} - 
\partial^2_{x_2})}f \big| \,                              \Big \|_{L^4(\R^n)} \leq C_T \Big\| D_x^{1/2} f \Big\|_{L^2(\R^2)},
\end{equation}
proved in \cite[Theorem 2.6]{RVV} and observe that the same argument used there
 to establish \eqref{max-non-ellip} shows that
\[
\Big \| \sup_{0\leq t\leq T}  \big|e^{it \Delta_{\rm H} }f \big| \, \Big\|_{L^4(\R^n)} \leq 
C_T \Big\| D_x^{n/4} f  \Big\|_{L^2(\R^n)},
\]
which is the desired estimate. The result then follows along the same lines as given in the proof of Theorem \ref{thm:main}.
\end{proof}

An important consequence of this is the possibility of dispersive blow up for the Davey-Stewartson system in the \lq\lq elliptic/elliptic" or \lq\lq hyperbolic/elliptic" cases (see \cite{GS} for more details 
about  these systems, in particular for theory of local well-posedness).  This system arose originally 
as an approximate description of surface gravity-capillary waves in shallow water, but has 
other applications as well.   

\begin{corollary}
For  $\alpha$ and $\beta$ real and non-zero, consider the Davey-Stewartson system
\begin{equation}
\label{DS-ivp}
\left \{ 
\begin{aligned} 
&\ i\partial_t u\pm \partial^2_{x_1}u + \partial^2_{x_2} u= \alpha |u|^2 u + \beta u \partial_{x_1}\phi ,\;\;\;\;\;\;\;\;\;x=(x_1,x_2)\in \R^2, \\
& \ \Delta \phi = \partial_{x_1} |u|^2.
\end{aligned}
\right.
\end{equation}
Then there exist initial values $u_0\in H^s(\R^2)\cap L^{\infty}(\R^2)\cap C^{\infty}(\R^2)$ with $s\in (\frac{1}{2},1]$ 
such that the solution $u$ of \eqref{DS-ivp} with the initial condition $u(x,0)=u_0(x)$ exhibits dispersive blow up.
\end{corollary}

\begin{proof}
Rewrite the system \eqref{DS-ivp} as the single equation
\begin{equation*}
i\partial_t u\pm \partial^2_{x_1}u + \partial^2_{x_2} u= \alpha |u|^2 u - \beta u \partial_{x_1}^2(-\Delta)^{-1}(|u|^2),
\end{equation*}
in the usual way.   The latter equation has the  equivalent form
\begin{equation*}
i\partial_t u\pm \partial^2_{x_1}u + \partial^2_{x_2} u= \alpha |u|^2 u - \beta u R_1R_1 (|u|^2),
\end{equation*}
where
\begin{equation*}
R_1 f(x_1,x_2):= \mathcal F^{-1}\left(i\frac{\xi_1}{|\xi|}\widehat f(\xi_1,\xi_2)\right)(x_1,x_2)
\end{equation*}
 denotes the $1$-Riesz transform in $\R^2$

The $L^p$-continuity of the Riesz transform, 
implies that the result in Proposition \ref{prop:CW} and the argument entailed in the proof of
 Proposition \ref{prop:key} still hold. Hence,  Theorem 
\ref{thm:main} extends to solutions of the system \eqref{DS-ivp}. \end{proof}


\section{Dispersive blow up in the Gross-Pitaevskii  equation}\label{sec:GP}

In this section, the discussion is moved to the initial-value problem 
\begin{equation} \label{eq:GP} 
i \partial_t \psi + \Delta \psi+ (1-|\psi|^2)\psi=0, \quad \psi\big|_{t=0} = \psi_0(x)
\end{equation}
for the Gross-Pitaevskii  equation.  
Here $(x,t) \in \R^n\times \R$ and $\psi$ is subject to the boundary condition 
\begin{equation}\label{eq:bc}
\lim_{|x|\to \infty} \psi(x,t) = 1, \quad \text{for all $t\in \R$.}
\end{equation}
The Gross-Pitaevskii equation arises, for example, in the description of Bose-Einstein condensates, superfluid helium $\text{He}^2$ and, 
in one spatial dimension, as a model for light propagation in a fiber 
optics cable (see for instance the survey article \cite{BGS3} and other articles in the same volume). 
\begin{remark} An important conserved quantity of \eqref{eq:GP} is the Ginzburg-Landau energy, defined by
\[
E(\psi) = \frac{1}{2} \int_{\R^n} |\nabla \psi|^2 \, dx + \frac{1}{4} \int_{\R^n} (1 - |\psi|^2)^2 \, dx.
\]
This invariant points to a natural energy space for the Gross-Pitaevskii equation, namely 
\[
\boE(\R^n) = \{ \psi \in H^1_{\text{loc}}(\R^n),  E(\psi) < + \infty \}
\]
 (see \cite{Ge1, Ge2} for more details).
\end{remark}


\subsection{A reformulation of the Gross-Pitaevskii Equation}

Due to the non-zero boundary condition \eqref{eq:bc} at infinity, the 
ideas that worked in the earlier sections do not apply directly to show 
dispersive blow up.    
To prove dispersive blow up for the Gross-Pitaevskii equation, make the change 
 \[\psi(t,x) =1+u(t,x),\] 
of the dependent variable so that $u\in L^2(\R^n)$ describes the deviation from the steady state. In terms of $u$, the 
Gross-Pitaevskii equation \eqref{eq:GP} becomes
\begin{equation}\label{bife}
i\partial_t u +\Delta u-2 \re \,u=F(u) , \quad  u \big|_{t=0} = \psi_0(x)-1,
\end{equation} 
where 
\[
F(u) = u^2+2|u|^2+|u|^2u.
\]
Even for this  transformed initial-value problem \eqref{bife}, dispersive blow up
 does not fall directly to the general lines of 
argument proposed earlier.  This is due to the appearance of 
the slightly odd, $\R$-linear term $-2 \re\,  u$, which at least
 in principle might cancel out the effect of dispersive blow up stemming from the Laplacian. 

This obstacle will be surmounted by using the further reformulation  developed 
by Gustafson, Nakanishi and Tsai in \cite{GNT1, GNT2}. 
Following their lead,  introduce the 
Fourier multipliers
\begin{equation}  \label{operators}
A:=  \sqrt{-\Delta(2-\Delta)} \quad {\rm and} \quad  B:= \sqrt{-\Delta(2-\Delta)^{-1}}\, .
\end{equation}
These operators satisfy $A = -\Delta B^{-1} = (2- \Delta)B=\sqrt{-\Delta(2-\Delta)}$. Next,  define the $\R$-linear operator
\[
\Upsilon u := B\, \re\, u + i \,\im \, u\equiv B u_1 + i u_2, 
\]
where $u = u_1 + iu_2$.   
One checks (see \cite{GNT1, GNT2}) that the left-hand side of the equation in \eqref{bife} 
can be written in the form
\begin{equation}\label{eq:identity}
i\partial_t u +\Delta u-2 \re \,u =  \Upsilon (i\partial_t  - A ) \Upsilon^{-1}u.
\end{equation}
Denote by $v$ the function
\[
v:=\Upsilon^{-1}u \equiv \Upsilon^{-1}(u_1+iu_2)=B^{-1}u_1+iu_2
\] 
and rewrite \eqref{bife} as
\begin{equation}\label{eq:reGP}
i \partial_t v - A v = \Upsilon^{-1}F(u), \quad v\big|_{t=0} = \Upsilon^{-1} (\psi_0-1).
\end{equation}
The associated free evolution is 
\[
w(x,t)=e^{-itA} v_0 (x) := \frac{1}{(2\pi)^n} \int_{\R^n} e^{it\sqrt{|\xi|^2(|\xi|^2+2)}} \widehat u_0(\xi) e^{i \xi\cdot x} \, d\xi,
\]
which can be represented as follows.
\begin{lemma}\label{lem:decomp}
For any $f\in L^2(\R^n)$ and any $t\not =0$, the group $\{e^{-itA}\}_{t\in \R}$ has the representation
\[e^{-itA} f(x) = (G(\cdot , t) \ast f)(x),\] where the kernel is $G =  G_1 + G_2$ with
\[
G_1(x,t)   =\frac{e^{it}}{(4\pi it)^{n/2}}\, e^{\frac{i|x|^2}{4t}}\quad and \quad G_2(x,t) := \frac{e^{it} }{(2\pi)^n} \int_{\R^n} \kappa(\xi, t) e^{it |\xi|^2} e^{ix\cdot \xi}d\xi.
\]
Here, the convolution is with respect to the spatial variable over $\R^n$ 
and the kernel  $\kappa $ is 
\[
\kappa(\xi, t)=2i\sin \Big(\frac{t}{2}r(\xi)\Big)e^{i\frac{t}{2} r(\xi)}
\]
with 
\[
r(\xi)= \frac{-2|\xi|^2}{\left(\sqrt{|\xi|^2(|\xi|^2+2)}+|\xi|^2\right)^2}\sim O(|\xi|^{-2}), \ \text{as $|\xi|\to \infty$.}
\]
Moreover, $r$ lies in $C_{\rm b}(\R^n)$ and is smooth away from the origin.
\end{lemma}

Note that $G_1$ is the usual Schr\"odinger group multiplied by $e^{it}$.  Remark also that, as $|\xi|\to \infty$, $\kappa(\cdot, t)$ decays to zero like $|\xi| ^{-2}$, 
uniformly on compact time-intervals. 

\begin{proof} 
The convolution kernel    $G$ is
\[G(x,t)=\mathcal F^{-1}\left(e^{it \sqrt{(|\xi|^2(|\xi|^2+2)}}\right)(x,t).\]
Observe that 
\[
\sqrt{(|\xi|^2(|\xi|^2+2)}=|\xi|^2+a(\xi),
\]
where 
\[
a(\xi)=\frac{2|\xi|^2}{\sqrt{(|\xi|^2(|\xi|^2+2)}+|\xi|^2}=1-\frac{2|\xi|^2}{\left(\sqrt{(|\xi|^2(|\xi|^2+2)}+|\xi|^2\right)^2}=1+r(\xi).\]
 Consequently, it follows that
 \[
 e^{it \sqrt{(|\xi|^2(|\xi|^2+2)}}=e^{it|\xi|^2}e^{it}e^{itr(\xi)}=e^{it}e^{it|\xi|^2}(1+f_t(|\xi|)),
 \]
where
\[
f_t(|\xi|)=2i\sin \Big(\frac{t}{2}r(\xi)\Big)e^{it\frac {r(\xi)}{2}}\]
is continuous, smooth on $\R^n$ and
decays to zero like $|\xi| ^{-2}$, as $|\xi|\to \infty$, uniformly on compact time intervals in $(0,\infty)$,  since
$r({|\xi|})$ does so.
This allows  the propagator in Fourier space to be written as
\[
\widehat G(\xi, t)= e^{it\sqrt{|\xi|^2(|\xi|^2+2)}}=e^{it}e^{it|\xi|^2}(1+\kappa(\xi, t))
\]
with $\kappa$ as above. 
\end{proof}

With this representation in hand, dispersive blow up for the evolutionary group $\{e^{-itA}\}_{t\in \R}$ 
can be established for at least  spatial dimensions less than or equal to three. 

\begin{lemma}\label{lem:DBUv}
Let $n\leq 3$  and suppose given $x_*\in \R^n$, $t_*>0$ and $s \in [0,\frac{n}{2})$.  Then 
there exist initial data $v_0\in C^\infty(\R^n)\cap H^s(\R^n)\cap L^\infty(\R^n)$ such that
\[
w(\cdot, t) = e^{-itA} v_0 \in H^s(\R^n)
\]
exhibits dispersive blow up at $(x_*, t_*)$.
\end{lemma}
\begin{proof} Choose $(x_*,t_*) = (0,\frac{1}{4})$ without loss of generality. 
	As in Lemma \ref{lem:linear}, let 
\begin{equation} \label{initdata}
v_0(x)=\frac{e^{-i|x|^2}}{(1+|x|^2)^m}, \quad \text{with  $\frac{n}{4}<m\leq \frac{n}{2}$.}
\end{equation}
Using Lemma \ref{lem:decomp},  it is found that 
\[
e^{-itA} v_0(x) = (G_1(\cdot , t) \ast v_0)(x) +(G_2(\cdot , t) \ast v_0)(x).
\]
In view of the calculations in Section \ref{sec:lin} it is inferred 
that the first term on the right-hand side will exhibit dispersive blow up at $(0,\frac{1}{4})$.
It thus suffices to show that  $G_2(\cdot,t)\ast v_0\in C_{\rm b}( \R^n)$, uniformly on compact time-intervals. 

The kernel $G_2(\cdot ,t)\in L^2(\R^n)$ for $n\leq 3$, due to the decay properties of $r(\xi)$. Since $v_0\in L^2(\R^n)$ by 
construction,  the product $\widehat G_2 (\cdot , t) \widehat v_0\in L^1(\R^n)$, uniformly on compact time-intervals.  The Riemann-Lebesgue lemma then implies the desired result.
\end{proof}

Note that even though $e^{-it A}$ represents the linear evolution operator associated to \eqref{eq:reGP}, it includes part of the nonlinearity of the 
original Gross-Pitaevskii equation \eqref{eq:GP}, as can be seen from \eqref{eq:identity}.


\subsection{Proof of dispersive blow up for Gross-Pitaevskii}

Lemma \ref{lem:DBUv} together with the results of Section \ref{sec:NLS} now allow us to prove 
dispersive blow up for equation \eqref{eq:reGP} and, consequently, also for the original 
Gross-Pitaevskii equation (in physically relevant dimensions $n = 1,2,3$). 
To effect a proof, the following lemma on the  mapping properties of $B$ and its inverse will be used.
\begin{lemma}\label{mapping}
The operator $B$ in \eqref{operators} may be written in the form
\[B={\rm Id}+B_1, \quad B^{-1}={\rm Id} +B_2,\] 
where $B_1$, $B_2\in \mathcal L(H^s(\R^n),H^{s+2}(\R^n))$, for all $s\in \R.$ 
\end{lemma}

\begin{proof}
The proof is a simple computation on the level of the Fourier symbols $b_1$ and $b_2$ associated to $B_1$ and $B_2.$ For example, the symbol $b_1$ is
\[
b_1(\xi)=-\frac{2}{(2+|\xi|^2)( \sqrt{|\xi|^2(2+|\xi|^2)^{-1}}+1)}.
\]
From this formula, one sees that $b_1$ is  in fact uniformly bounded and 
is $O(|\xi|^{-2})$ as $|\xi| \to \infty$. Similarly, one finds
\[b_2(\xi)=\frac{2}{(2+|\xi|^2)\sqrt{|\xi|^2(2+|\xi|^2)^{-1}}[1+(|\xi|^2(2+|\xi|^2)^{-1}]^{1/2}},\]
from which one infers the assertion about $B_2$.
\end{proof}

In particular, since $\Upsilon^{-1}u=B^{-1}u_1+iu_2$, the last lemma
 implies that, for all $s\in \R, \Upsilon$ and  $\Upsilon ^{-1}$ both lie in 
$\mathcal L(H^s(\R^n),H^{s}(\R^n))$. 

Here is  the main result of this section.

\begin{theorem} \label{DBU:v}
Let $n\leq 3$, $x_*\in \R^n$, $t_*>0$ and $s \in (0,\frac{n}{2})$ be given.  Then
there exist initial data $v_0\in C^\infty(\R^n)\cap H^s(\R^n)\cap L^\infty(\R^n)$ 
such that the corresponding solution $v\in C([0,T], H^s(\R^n))$ of \eqref{eq:reGP} exhibits dispersive blow up at $(x_*,t_*).$

\end{theorem}
\begin{proof} Take $v_0$ in the form  \eqref{initdata} in the proof of Lemma \ref{lem:DBUv}. 
By an approprite choice of $m$, Lemma \ref{lem:linear},  guarantees that $v_0$ lies 
in $H^s(\R^n)$.  
Next, in view of Lemma \ref{mapping}, the corresponding $u_0=\Upsilon^{-1} v_0\equiv B^{-1}\text{Re} \, v_0 + i \text{Im} \, v_0$ has the same Sobolev regularity as $v_0$. 
For such a $u_0$, the Cazenave-Weissler theory can be applied to the Gross-Pitaevskii equation \eqref{bife}, yielding a local in-time solution $u\in C([0,T], H^s(\R^n))$ for the specified
 $s<\frac{n}{2}$, which  
 leads to a solution $v\equiv \Upsilon u \in C([0,T], H^s(\R^n))$ of \eqref{eq:reGP}. 
Consequently,  we can use the strategy followed for the usual 
nonlinear Schr\"odinger equation. 
To this end, write \begin{equation}\label{eq:duhamelv}
v(x,t)=e^{-it A}v_0 (x)+ i \int_0^te^{-i(t-s)A}  \Upsilon^{-1}F(u(x,s)) \, ds,
\end{equation}
where the nonlinearity is explicitly given by 
\[
 \Upsilon^{-1}F(u) =  B^{-1} (3 u_1^2+u_2^2+|u|^2u_1)+iu_2(2u_1+|u|^2),
\]
since $u=u_1+iu_2$. As before, the first term on the right-hand side exhibits dispersive blow up at $(x_*, t_*)=(0,\frac{1}{4})$ provided $n<4$. 
The advertised result will be in hand when the second term is known to be uniformly bounded.

The integral term in \eqref {eq:duhamelv} splits into $I_1+I_2$ where
\[
I_j(x,t)=\int_0^t G_j(\cdot, t-s)\ast \Upsilon^{-1}F(u(\cdot,s)) \, ds,\quad j=1,2,
\]
corresponding to $G=G_1+G_2$.
The fact that $I_2$ is a bounded continuous function can be concluded by the same argument 
that prevailed in the proof 
of Lemma \ref{lem:DBUv}. For $n<4$, $G_2(\cdot, t)\in L^2(\R^n)$, and so is $\Upsilon^{-1} F(u)$ 
because
 $u\in H^s(\R^n)$ and $\Upsilon^{-1}\in \mathcal L(H^s(\R^n),H^{s}(\R^n))$.  

On the other hand, $I_1$ is given by
\[
 I_1(x,t)=\int_0^t G_1(\cdot, t-s)\ast \big( B^{-1}[3u_1^2+u_2^2+|u|^2u_1]+u_2(2u_1+|u|^2)\big) (\cdot,s) \, ds.
 \]
Inasmuch as  $G_1$ is, up to a multiplicative constant, the fundamental solution of the usual linear Schr\"{o}dinger equation 
(see Lemma \ref{lem:decomp}), the double  integral $I_1$ 
is of the same form (up the the appearance of $B^{-1}$) as the usual nonlinearity in the Duhamel representation of the nonlinear Schr\"odinger equation.
In view of the mapping properties of $B^{-1}$, 
the desired result of boundedness and continuity 
then follows from the corresponding proof for the usual nonlinear Schr\"odinger  equation 
given in Sections \ref{sec:DBUNLS} and \ref{sec:global}.
\end{proof}


As a corollary, we infer the appearance of dispersive blow up for the original 
Gross-Pitaevskii equation \eqref{eq:GP} in physically relevant dimensions.


\begin{corollary} Let $n\leq 3$. Given $x_*\in \R^n$, $t_*>0$, there exist smooth and bounded initial data $\psi_0 \in C_{\rm b}(\R^n)$ with 
$\psi_0-1\in L^2(\R^n)$, such that the solution $\psi$ exhibits 
dispersive blow up at $(x_*, t_*)$.
\end{corollary}

\begin{proof}  To establish dispersive blow up for \eqref{eq:GP},  the results for $v$ 
need to be transferred back to the
variable $\psi$.   In search of such a conclusion, note that the initial data in \eqref{eq:reGP} 
is given by
\[
v_0(x) = \Upsilon^{-1} (\psi_0-1) = B^{-1}(\re \, \psi_0 -1) + i \, \im \, \psi_0.
\]
The specific choice $v_0(x)= \frac{e^{-i|x|^2}}{(1+|x|^2)^m}$ then corresponds to the  initial data 
\[
\psi_0 (x) = 1+ B \, \re \, v_0(x)  + i \, \im \, v_0(x) = 1+ B \left(\frac{\cos |x|^2}{(1+|x|^2)^m} \right) +  \frac{i \sin |x|^2}{(1+|x|^2)^m} 
\]
for \eqref{eq:GP}.  Since $B:L^2(\R^n)\to L^2(\R^n)$, it is clear from the last formula 
that $\psi_0 - 1 \in L^2(\R^n)$.  By Lemma \ref{mapping}, $B={\rm Id}+ B_1$ with  $B_1\in \mathcal L(L^2(\R^n), H^2(\R^n)) $ and $H^2(\R^n)\subset C_{\rm b}(\R^n)$ for $n\leq 3,$ so 
it is concluded that $\psi_0\in C_{\rm b}(\R^n).$

Now, if $v=v_1+iv_2$ is a solution of  \eqref{eq:reGP} with the dispersive blow up 
property at a point $(x_*,t_*)$ provided by Theorem \ref{DBU:v}, 
the corresponding $u =u_1+iu_2$ is given by
\[u_1=B^{-1}v_1=(I+B_2)v_1,\quad u_2=v_2.\]
Since $B_2$ is a smoothing operator, $u,$  and thus $\psi=1+u$ satisfies the DBU property at the same point $(x_*,t_*).$  
\end{proof}

\section{Higher-order nonlinear Schr\"odinger equations} \label{sec:higherNLS}

In this final section, we indicate how  results of dispersive blow up can be extended 
to higher-order nonlinear Schr\"odinger equations.  It is  mathematically 
natural to inquire whether or not higher-order terms destroy dispersive blow up, 
but the practical motivation for considering such an extension is perhaps even more 
telling.   In nonlinear optics, third and fourth-order Schr\"odinger-type 
equations frequently appear in the description of various  wave phenomena. 
In particular, the analysis of optical rogue-wave formation has been based on higher-order 
nonlinear Schr\"odinger equations  (see, for example, \cite{D, DGE, Mu, T1}).

As the ideas and even much of the technical detail mirror closely what has gone
before, we content ourselves with admittedly sketchy indications of how the theory
is developed.   The one point which would require serious new effort has to do 
with an appropriate generalization of the Cazenave--Weissler theory in \cite{CW} 
to a higher-order setting.  This is not attempted here, but is deserving of 
further investigation at a later stage.

\subsection{Fourth-order nonlinear Schr\"odinger equation}
In this subsection,  initial-value problems for 
fourth-order nonlinear  Schr\"{o}dinger equations of the form 
\begin{equation}\label{eq:4th}
i\partial_t u +\alpha \Delta u+\beta \Delta ^2 u+\lambda |u|^pu=0, \quad u\big|_{t=0}=u_0(x),
\end{equation}
 are considered.  Here, the parameters  $\alpha, \beta, \lambda$ are real constants, 
 with $\beta \neq 0$.  
Theory for this initial-value problem can be found, for example, in \cite{FES, P} and in 
the references cited in these works.   
If $\alpha =0$, the partial differential equation
 is often referred to as the  bi-harmonic NLS equation  (see, e.g., \cite{BFM}). 
 A simple scaling allows us to assume $\beta=1$ and to consider only the values 
$\alpha \in  \lbrace 0, -1, +1\rbrace$, though time may need to be reversed.

To establish dispersive blow up for \eqref{eq:4th},  the dispersive properties 
of the associated linear equation
\begin{equation}\label{eq:lin4th}
i \partial_t u -\alpha  \Delta u+ \Delta^2 u=0, \quad \alpha \in \lbrace 0, -1, +1\rbrace
\end{equation}
are helpful, just as for the lower-order cases.  As should be  clear from the preceding theory, 
 the possibility of dispersive blow up for \eqref{eq:lin4th} is  linked to 
the dispersive  properties of the fundamental solution 
\begin{equation}\label{4disp}
\Sigma_\alpha (x,t)= \frac{1}{(2\pi)^{n/2} } \int_{\R^n} e^{it(|\xi|^4+\alpha |\xi|^2)+ix\cdot \xi} \, dx
\end{equation}
of \eqref{eq:lin4th}, which have in fact been established already in \cite{BAKS}.

\begin{lemma}\label{4dec}
Let $\Sigma_\alpha$ be as in \eqref{4disp} and $\mu \in \N^n$a multi-index. 
\begin{enumerate}
\item If $\alpha  =0$, there exists a $C>0$ such that for $x \in \R^n$ and $t> 0$,
$$
|\partial^\mu \Sigma_0(x,t)|\leq Ct^{-(n+|\mu|)/4}\left(1+\frac{|x|}{t^{1/4}}\right)^{-(n-|\mu|)/3}.
$$
\item 
For $t > 0$ and  either $t\leq 1$ or $|x|\geq t$,  there exists a $C>0$ such that
$$
|\partial^\mu \Sigma_\alpha(x,t)|\leq 
Ct^{-(n+|\mu|)/4}\left(1+\frac{|x|}{t^{1/4}}\right)^{-(n-|\mu|)/3}
$$
for $\alpha = \pm 1$.
\end{enumerate}
\end{lemma}

Strichartz estimates then follow pretty much directly from Lemma \ref{4dec}. 
 In some detail, we say that 
the pair $(q,r)$ is admissible for the fourth-order Schr\"odinger group $\{e^{it (\Delta^2-\alpha \Delta)}\}_{t\in \R}$ if
\begin{equation}\label{eq:adm}
\frac{1}{q}=\frac{n}{4}\left(\frac{1}{2}-\frac{1}{r}\right),
\end{equation}
for $2\leq r\leq \frac{2n}{n-2}$ if $n \geq 3$, respectively, $2\leq r\leq \infty$ if $n=1$ 
and  $2\leq r<\infty$ if $n=2$. 
Using this, one has the following estimates, which are the fourth-order counterpart to the ones 
reported in Lemma \ref{lem:Strich}. In what follows, $T=+\infty$ when 
$\alpha =0$ and $T$ is any nonnegative number when $\alpha= \pm 1.$
 
\begin{lemma}[\cite{BAKS}]
\label{Str}
Let $(q, r)$ be admissible in the sense of \eqref{eq:adm}. Then there exists $c=c(n, r,T)$ such that
\[  \big \|  e^{it (\Delta^2-\alpha \Delta)} f \big \|_{L^q((-T,T);L^r(\R^n))}\leq c\|f \|_{L^2(\R^n)}. \]
The linear operator 
\[
\Phi f= \int_0^t e^{i(t-s)(\Delta^2-\alpha \Delta)}f(s)ds
\]
is bounded in the sense that 
\[ \| \Phi f\|_{L^q((-T,T);L^r(\R^n))}\leq c\|f\|_{L^{q'}((-T,T);L^{r'}(\R^n))}, \]
where $\frac{1}{q}+\frac{1}{q'}=1$ and  $\frac{1}{r}+\frac{1}{r'}=1.$ 
\end{lemma}

If we assume for the moment that dispersive blow up holds true for the linear model \eqref{eq:lin4th}, then the 
Strichartz estimates above are already sufficient to prove dispersive blow up for the nonlinear equation \eqref{eq:4th} in the physically relevant dimensions $n\leq 3$.

\begin{proposition}
Let $n\leq 3$ and $\alpha \in \lbrace 0, -1, +1\rbrace$. Assume that the linear fourth-order equation \eqref{eq:lin4th} exhibits dispersive blow up 
at some  point $(x_*, t_*)$ in space-time. Then, for $p < \frac{8}{n}-1$, so does the fourth-order 
initial-value problem \eqref{eq:4th}.
\end{proposition}

\begin{proof}
The proof follows closely the one given in \cite{BS2} for the one-dimensional, second-order   nonlinear Schr\"{o}dinger equation. 
In particular, the fact that the dispersive estimate in Lemma \ref{4dec} of the fundamental solution $\Sigma_\alpha$ has temporal behavior that goes like  $t\mapsto t^{-n/4}$, which is locally integrable for $n<4$, is a key point in the proof. 

The Duhamel representation of \eqref{eq:4th} is given by
\begin{align*}
u(x,t) = & \, e^{i t (\Delta^2-\alpha \Delta)} u_0(x) + i \lambda \int_0^t e^{i(t-s)(\Delta^2-\alpha \Delta)} |u(x,s)|^p u(x,s) \, ds \\
= & \,  e^{i t (\Delta^2-\alpha \Delta)} u_0(x) + i \lambda I (x,t).
\end{align*}
The first term on the right-hand side exhibits dispersive blow up by assumption.  
To prove that the integral term 
is continuous and bounded, notice that  
\[
|I(x,t) |\leq  C\int_0^t\int_{\R^n} \frac{1}{(t-s)^{n/4}}|u|^{p+1}(x-y,t-s)\, ds\, dy 
\]
using Lemma \ref{4dec}. Applying H\"{o}lder's inequality, with a $\gamma\in(0,4/n)$ to be determined presently, it is found that
\[
|I(t,x) | \leq\left(\int_\R \frac{ds}{(t-s)^{n\gamma/4}}\right)^{1/\gamma}\left(\int_\R\|u(\cdot,s)\|_{p+1}^{\gamma'(p+1)} ds\right)^{1/\gamma'},
\]
with $\frac{1}{\gamma}+\frac{1}{\gamma'}=1$. 
Choose an admissible Strichartz pair in the range \eqref{eq:adm} as follows.
Take $r=p+1$ so that $q=\frac{8(p+1)}{n(p-1)}.$ The condition $\gamma'(p+1)\leq q$ then yields
\[\gamma'=\frac{\gamma}{\gamma -1} \leq  \frac{8}{n(p-1)}.\] 
Combined with the condition $\gamma <\frac{4}{n}$, one obtains
\[
p < \frac{8}{n} \left( 1 - \frac{1}{\gamma}\right) = \frac{8}{n} -1,
\]
and the assertion is proved.
\end{proof}

\begin{remark}
The strategy deployed in this proof does not yield the optimal range of exponents $p$ nor is 
it valid for $n \geq 4$.  This is because 
it is based only on Strichartz estimates 
and because $t\mapsto t^{-n/4}$ is locally integrable only  for $n<4$. 
To extend the proof to higher dimensions $n > 3$, and to more general 
nonlinearities $p>0$, one could argue as in the proof of Proposition \ref{prop:key}.  However,
an essential ingredient in our argument was the Cazenave-Weissler result 
recounted in  Proposition \ref{prop:CW}.
Thus, to carry out this line of reasoning successfully, we would need 
 the analog of the Cazenave-Weissler results  in the case of fourth-order equations, as well as  
the corresponding smoothing estimates for the Duhamel term established in Section \ref{sec:global}. These tasks will be the goal of an upcoming work.
\end{remark}

To close the analysis, it is still required to prove that the linear fourth-order equation \eqref{eq:lin4th} does exhibit dispersive blow up. 
To this end, we will need a more precise description of the decay of the fundamental solution. 
This will only be carried out in the one-dimensional case,
 though the result in higher spatial dimensions does not require new ideas, 
just more complex calculations.

\begin{proposition}\label{dbu4th}
Let $n=1$. Given $(x^*,t^*)\in \R\times (\R\setminus \lbrace 0\rbrace)$, there exists $u_0 \in C^\infty(\R)\cap L^2(\R)\cap L^\infty(\R)$ such that the 
solution $u \in C_{\rm b}(\R; L^2(\R))$ to \eqref{eq:lin4th} has the following properties.
\begin{enumerate}
\item The solution $u$ blows up at $(x_*,t_*)$ which is to say
\[
\lim _{(x,t)\in \R\to (x_*, t_*)} | u(x,t)| = + \infty.
\]
\item The function $u$ is continuous on $\big\{(x,t)$ on $\R\times \R \setminus \{t_*\}\big\}$.
\item  The solution $u(\cdot, t_*)$ is a continuous function on $\R \setminus \{x_*\}$.
\end{enumerate} 
\end{proposition}

\begin{proof} For $n=1$, it is readily checked that the fundamental solution of 
\begin{equation}\label{4th1D}
i \partial_t u_t-\alpha \partial_x^2 u + \partial_x^4u =0
\end{equation}
is given by
\[
\Sigma_\alpha(x,t)=\frac{1}{(4t)^{1/4}}B\left(2\alpha t^{1/2},\frac{x}{(4t)^{1/4}}\right)
\]
where  
\begin{equation*}\label{pear}
B(x,y)=\frac{1}{\pi} \int_\R e^{i\big(\frac{1}{4}s^4+\frac{1}{2}xs^2+ys\big)}\, ds
\end{equation*}
 is the  Pearcey integral. The Pearcey integral is a smooth and bounded function of $(x,y) \in \R^2$ which decays in both variables thusly: 
\begin{equation*}\label{asPe}
|B(x,y)|\leq c\big(1+y^2+|x|^3\big)^{-1/18}\Big(1+(1+y^2+|x|^3)^{-5/9}|(3y)^2+(2x)^3|\Big)^{-1/4}
\end{equation*}
(see, e.g., \cite{BAKS}).
To establish dispersive blow up for \eqref{4th1D}, an asymptotic expansion of the Pearcey integral 
with respect to the second variable $y$ is helpful. Such an expansion has been established in \cite{Pa}, Formulas (2.14) and (5.12).
\footnote{Note the rotation of coordinates in those formulas as compared to the way they 
are written here.} These results imply that
\begin{itemize}
\item[(i)] in the case $\alpha =0$, one has
\begin{equation*}\label{Pe0}
B(0,y)=C_1y^{-1/3}+C_2 y^{-5/3}+O(|y|^{-3}),
\end{equation*}
 as $|y|\to +\infty$, uniformly for bounded values of $x$, where $C_1$ and $C_2$ are nonzero complex constants and
\item[(ii)] when $\alpha \neq0$, 
\begin{align*}\label{Pe2}
B(x,y)=2^{1/6}e^{-i\pi/24}y^{-1/3}e^{-i x^2/6}
\exp\Big(-\frac{3i}{4^{4/3}}y^{4/3}+\frac{i}{4^{2/3}}xy^{-2/3}]\Big)\times \\
\left(1+\frac{4^{-1/3}}{3}xy^{-2/3}(1-\frac{i}{9}x^2)+O(|y|^{-4/3})\right)
\end{align*}
for $|y|\to +\infty$ and bounded $x$.
\end{itemize}

With these results in hand, an argument can be mounted that mimics
 the case of the usual Schr\"odinger group. Without loss of generality, take it 
that  $(x_*,t_*) =(0,\frac{1}{4})$ and consider in \eqref{4th1D} the initial data 
\[
u_0(x)=\frac{\Sigma_0(-x,\frac{1}{4})}{(1+x^2)^m}=\frac{B(0,-x)}{(1+x^2)^m},\quad \frac{1}{12}<m\leq \frac{1}{6},
\]
if $\alpha =0,$ and
\[
u_0(x)=\frac{\Sigma_\alpha(-x,\frac{1}{4})}{(1+x^2)^m}=\frac{B(\alpha,-x)}{(1+x^2)^m},\quad \frac{1}{12}<m\leq \frac{1}{6},
\]if $\alpha =\pm 1$.  

The asymptotics of the Pearcey integral can then be brought to bear to prove Proposition \ref{dbu4th}.  In particular, one uses these asymptotics to show that the solution of the linear problem blows 
up at $(0,\frac14)$, that it is continuous on the 
complement of  this point and that the relevant Duhamel integral is everywhere continuous. 
We pass over the details.
\end{proof}

The analogous result for general dimensions $n\in \N$ can be established 
by passing to radial coordinates and reducing it to the one-dimensional case (as in \cite{BAKS} where the first term in the asymptotic expansion is given).


\subsection{Third order NLS in one dimension}

Finally, we briefly discuss the situation for a third-order nonlinear Schro\"odinger 
equation in $n=1$ dimensions. Consider  the initial=value problem
\begin{equation}\label{HNLS}
\partial_t u +i\alpha \partial^2_{x} u +\beta \partial^3_{x} u +i\gamma |u|^2u=0, \quad u\big|_{t=0}=u_0(x),
\end{equation}
where $x\in \R$, and $\alpha, \beta, \gamma \in \R \setminus \{ 0 \}$ are  given parameters. 
This is a  model problem that arises in optical wave propagation (see \cite{Mu, T1}). 

Equation \eqref{HNLS} appears similar to the well known Korteweg-de Vries equation, hence the 
appearance of dispersive blow up can be established along the same lines as was 
pursued in \cite{BS1} for this latter equation.  Indeed, 
the associated linear equation,
\begin{equation}\label{linHLS}
\partial_t  u +i\alpha \partial^2_{x}  u  +\beta \partial^3_{x}  u =0,
\end{equation}
admits a fundamental solution of the form
\[
\Lambda(x,t) = \frac{1}{2\pi} \int_\R e^{i\beta \xi^3 - i\alpha \xi^2 +ix \cdot \xi} \, d\xi.
\]
The expansion 
\[
\left(\xi+\frac{\alpha}{3\beta}\right )^3=\xi^3+\frac{\alpha}{\beta} \xi^2+\frac{\alpha^2 \xi}{3\beta^2}+\frac{\alpha^3}{27 \beta^3}
\]
allows us to write $\Lambda$ in the form
\[
\Lambda(x,t)=\frac{1}{2\pi (t\beta)^{1/3}}\exp \left(\frac{4it\alpha^3}{27 \beta^2}\right)\exp\left(\frac{-i\alpha x}{2\beta^2}\right)\text{Ai}\left(\frac{1}{t^{1/3} \beta^{1/3}}
\Big(x-\frac{\alpha ^2}{3\beta}t\Big)\right).
\]
Here, $\text{Ai}$ denotes the well-known the Airy function defined, for example, as 
the improper Riemann integral 
\[
\text{Ai}(z)=\int_{\R}e^{i(\xi^3+iz\xi)}\, d\xi.
\]
In \cite{BS1}, the dispersive properties of the Airy function are the basis for establishing dispersive blow up for the Korteweg-de Vries equation.  
Following these ideas, dispersive blow up for \eqref{linHLS} at $(x_*,t_*)=(0,1)$ is easily obtained by taking initial data of the form
\[
u_0(x)=\frac{A(-x)}{(1+x^2)^m}, \quad \text{with $\frac{1}{8}<m\leq \frac{1}{4}$}
\]
and
\[
A(x)=\text{Ai}\left (\frac{\alpha^2+x}{3\beta^{4/3}}\right ).
\]
The proof of dispersive blow up for  this third-order  Schr\"odinger equation 
can then be accomplished just as in \cite{BS1}. 
Indeed, the proof is  easier since, contrary to the Korteweg-de Vries equation, 
the Duhamel representation of \eqref{HNLS} does 
not involve any spatial derivatives of the dependent variable.   Consequently 
it can be shown to be bounded by using only  Strichartz estimates for 
the linearized Korteweg-de Vries equation (see \cite{BAKS2, Bo}).

\begin{remark}  Combining  the results of the foregoing sections allows one to deduce
  dispersive blow up for 
nonlinear Schr\"odinger-type equations with anisotropic dispersion, such as
\[
i \partial_t u+\alpha \Delta u +i\beta \partial^3_{x_1} u+\gamma \partial^4_{x_1}u+|u|^pu=0,
\]
where $\alpha, \beta, \gamma\in \R \setminus \{ 0 \}$. The Cauchy problem for this equation 
has been studied in \cite{Bo} (see 
also \cite{FES}).
\end{remark}

\begin{merci}
We tender thanks to an anonymous referee for her/his {\it very} helpful remarks.
\end{merci}


\bibliographystyle{amsplain}

\end{document}